\documentclass[11pt,reqno]{amsart}
\usepackage[utf8]{inputenc}

\usepackage[margin=1in]{geometry}
\usepackage{amsmath}
\usepackage{amsfonts}
\usepackage{amssymb}
\usepackage{amsthm}
\usepackage{enumitem}
\usepackage{mathrsfs}
\usepackage{mathtools}
\usepackage{tikz-cd}
\usepackage{enumitem}

\usepackage[colorlinks=true,hyperindex, linkcolor=magenta, pagebackref=false, citecolor=cyan,pdfpagelabels]{hyperref}

\newtheorem{theorem}{Theorem}[section]
\newtheorem{prop}[theorem]{Proposition}
\newtheorem{lemma}[theorem]{Lemma}

\theoremstyle{definition}
\newtheorem{definition}[theorem]{Definition}
\newtheorem{example}[theorem]{Example}
\newtheorem{remark}[theorem]{Remark}

\newtheorem{conjecture}[theorem]{Conjecture}
\newtheorem{claim}[theorem]{Claim}

\renewcommand{\rm}{\mathrm}

\newcommand{\Spec}{\mathrm{Spec} \,}

\newcommand{\eps}{\epsilon}

\newcommand{\vp}{\varphi}

\newcommand{\bA}{{\mathbf A}}
\newcommand{\bB}{{\mathbf B}}
\newcommand{\bC}{{\mathbf C}}

\newcommand{\bF}{{\mathbf F}}
\newcommand{\bG}{{\mathbf G}}
\newcommand{\bH}{{\mathbf H}}

\newcommand{\bQ}{{\mathbf Q}}

\newcommand{\bT}{{\mathbf T}}

\newcommand{\bZ}{{\mathbf Z}}

\newcommand{\cO}{\mathcal{O}}

\newcommand{\cT}{\mathcal{T}}

\newcommand{\sD}{{\mathscr D}}

\newcommand{\sU}{{\mathscr U}}

\newcommand{\fm}{\mathfrak{m}}

\newcommand{\fp}{\mathfrak{p}}
\newcommand{\fq}{\mathfrak{q}}

\newcommand{\bGm}{\bG_{\rm{m}}}
\newcommand{\LG}{\prescript{L}{}{G}}
\newcommand{\Lvp}{\prescript{L}{}{\varphi}}

\DeclareMathOperator{\Ad}{Ad}

\DeclareMathOperator{\chara}{char}

\DeclareMathOperator{\coker}{coker}

\DeclareMathOperator{\diag}{diag}

\DeclareMathOperator{\Frac}{Frac}

\DeclareMathOperator{\GL}{GL}

\DeclareMathOperator{\GSp}{GSp}
\DeclareMathOperator{\Hom}{Hom}

\DeclareMathOperator{\im}{im}

\DeclareMathOperator{\Out}{Out}

\DeclareMathOperator{\PGL}{PGL}
\DeclareMathOperator{\QCoh}{QCoh}

\DeclareMathOperator{\Rep}{Rep}

\DeclareMathOperator{\SL}{SL}

\DeclareMathOperator{\Tor}{Tor}
\DeclareMathOperator{\tr}{tr}

\DeclareMathOperator{\uHom}{\underline{\Hom}}

\newcommand{\wtil}[1]{\wtil}

\newcommand{\ov}[1]{\overline{#1}}

\title{Connected components of the moduli space of L-parameters}
\author{Sean Cotner}

\begin{document}

\begin{abstract}
Recently, in order to formulate a categorical version of the local Langlands correspondence, several authors have constructed moduli spaces of $\bZ[1/p]$-valued L-parameters for $p$-adic groups. The connected components of these spaces over various $\bZ[1/p]$-algebras $R$ are conjecturally related to blocks in categories of $R$-representations of $p$-adic groups. Dat-Helm-Kurinczuk-Moss described the components when $R$ is an algebraically closed field and gave a conjectural description when $R = \ov{\bZ}[1/p]$. In this paper, we prove a strong form of this conjecture applicable to any integral domain $R$ over $\ov{\bZ}[1/p]$.
\end{abstract}

\maketitle

\section{Introduction}

Let $p$ be a prime number, let $F$ be a nonarchimedean local field of residue characteristic $p$, and let $G$ be a connected reductive group over $F$. Since its inception, one of the primary goals in the Langlands program has been to understand the smooth $\bC$-representations of $G(F)$. Recently, in order to study congruences of smooth representations, there has been considerable interest in the category $\Rep_R(G(F))$ of $R$-representations, where $R$ is any commutative ring, e.g., $R = \ov{\bZ}_\ell$; see \cite{Vigneras-ICM} for a recent and extensive guide to the literature. We will only consider the case that $p$ is invertible in $R$, and to avoid rationality issues we will assume further that $R$ is a $\ov{\bZ}[1/p]$-algebra. \smallskip

When studying any category of representations, a fundamental issue is to understand the blocks, i.e., the indecomposable direct factors. If $R = \bC$, the blocks of $\Rep_R(G(F))$ were determined by Bernstein \cite{Bernstein-Deligne} in terms of (super)cuspidal support. If $R = \ov{\bF}_\ell$, $\ell \neq p$, it is expected that a similar block decomposition holds, and there has been a good deal of recent work in this direction, for instance \cite{Secherre-Stevens-block}, \cite{Drevon-Secherre}, \cite{Cui-support}. When $R = W(\ov{\bF}_\ell)$ for $\ell \neq p$ and $G = \GL_n$, Helm established a similar block decomposition in \cite{Helm1}, \cite{Helm2}, \cite{Helm3}. Understanding these blocks for inner forms of $\GL_n$ has led to new results on the Jacquet--Langlands correspondence in \cite{Secherre-Stevens}.\smallskip

This paper concerns the ``dual side" of this question. To describe it, let $W_F$ be the Weil group of $F$, which acts on the Langlands dual $\hat{G}$ over $\bZ[1/p]$. Recently, \cite{DHKM} and \cite{Zhu} have constructed a space $\underline{\rm{Z}}^1(W_F^0, \hat{G})$ consisting (roughly) of 1-cocycles $\vp\colon W_F \to \hat{G}$ over $\bZ[1/p]$, or, equivalently, L-homomorphisms $\Lvp\colon W_F \to \LG \coloneqq \hat{G} \rtimes W_F$.\footnote{Because of the incompatibility of the $\ell$-adic and $\ell'$-adic topologies for $\ell \neq \ell'$, it is not totally clear how to make this definition functorially. In the approach of \cite{DHKM}, which we recall in Section~\ref{section:example}, this is achieved by replacing $W_F$ by a dense subgroup $W_F^0$ with ``discretized inertia".} Recent categorical versions of the local Langlands correspondence as in \cite{Hellmann}, \cite{BZCHN}, \cite{Zhu}, and \cite{Fargues-Scholze} suggest that there should be a fully faithful embedding of $\Rep_R(G(F))$ into $\QCoh([\underline{\rm{Z}}^1(W_F^0, \hat{G})/\hat{G}]_R)$. The latter category breaks up into a product indexed by the components of $\underline{\rm{Z}}^1(W_F^0, \hat{G})_R$, so one expects that each block of $\Rep_R(G(F))$ can be associated to such a component.\smallskip

Focusing on a single prime, i.e., taking $R = \ov{\bF}_\ell$ or $\ov{\bZ}_\ell$ for $\ell \neq p$, the connected components of $\underline{\rm{Z}}^1(W_F^0, \hat{G})_R$ were determined in \cite[4.8]{DHKM} in terms of restriction to prime-to-$\ell$ inertia. On the other end of the spectrum, instead of considering a single prime $\ell \neq p$, one can consider all primes other than $p$ at the same time by taking $R = \ov{\bZ}[1/p]$. Representations of finite $p$-groups do not deform away from characteristic $p$, so restriction to the wild inertia subgroup $P_F$ is a natural discrete invariant of $\underline{\rm{Z}}^1(W_F^0, \hat{G})_{\bZ[1/p]}$. In fact, \cite[\S 4.1]{DHKM} exhibits a decomposition
\begin{equation}\label{eqn:z-1/p-decomp}
\underline{\rm{Z}}^1(W_F^0, \hat{G})_{\ov{\bZ}[1/p]} = \coprod_{\substack{\phi \in \Phi \\ \alpha \in \Sigma(\phi)}} \underline{\rm{Z}}^1(W_F^0, \hat{G})_{\phi, \alpha},
\end{equation}
where membership of a cocycle $\vp$ in $\underline{\rm{Z}}^1(W_F^0, \hat{G})_{\phi, \alpha}$ is determined by the restriction $\phi$ of $\vp$ to $P_F$ and a map $\alpha$ from $W_F$ to the component group of a twisted centralizer of $\prescript{L}{}{\phi}$ in $\LG$. We will recall the precise definitions in a self-contained way in Section~\ref{ss:motivation}.

\begin{conjecture}\cite[4.3]{DHKM}, \cite[3.10]{Dat-notes}\label{conj:4.3}
    Each $\ov{\bZ}[1/p]$-scheme $\underline{\rm{Z}}^1(W_F^0, \hat{G})_{\phi, \alpha}$ is connected.
\end{conjecture}

\begin{theorem}\label{theorem:intro-MAIN}
    Conjecture~\ref{conj:4.3} is true.
\end{theorem}

Conjecture~\ref{conj:4.3} was previously proved in \cite[4.30]{DHKM} when $Z(\hat{G})$ is smooth, which corresponds to the condition that $G$ has fundamental group (in the sense of Borovoi) with no torsion of order prime to $p$. This handles the case that $G$ is semisimple and simply connected, as well as some other groups like $G = \GL_n$ or $\GSp_{2n}$, but it says very little if, for instance, $G$ is of adjoint type. Indeed, when $G = \PGL_n$ for an integer $n$ which is divisible by at least two distinct primes, this condition does not hold for any $p$.\smallskip

Roughly, Conjecture~\ref{conj:4.3} describes the way in which components of $\underline{\rm{Z}}^1(W_F^0, \hat{G})_{\bC}$ conspire to connect modulo the primes $\ell \neq p$. In fact, it can equivalently be phrased as the statement that if two $\bC$-valued cocycles $\vp, \psi$ have the same restriction to $P_F$ and the same invariant $\alpha$ as above, then there is a sequence of primes $\ell_1, \dots, \ell_n \neq p$ and a sequence $\vp = \vp_0, \vp_1, \dots, \vp_n = \psi$ such that the components containing $\vp_i$ and $\vp_{i+1}$ in $\underline{\rm{Z}}^1(W_F^0, \hat{G})_{\bC}$ meet modulo $\ell_{i+1}$ for all $0 \leq i \leq n-1$. Under this interpretation, we note the similarity to \cite[Thm.\ B]{Secherre-Stevens}. \smallskip

As stated above, the cases $R = \ov{\bZ}[1/p]$ and $R = \ov{\bF}_{\ell}$ are in a sense at opposite extremes. We will interpolate between them by considering an arbitrary integral domain $R$ over $\ov{\bZ}[1/p]$. For simplicity, let us assume for now that $R = \ov{\bZ}[D^{-1}]$ for some multiplicatively closed subset $D \subset \bZ$. Reasoning as in the case $R = \ov{\bZ}[1/p]$ suggests that there should be a discrete invariant of $\underline{\rm{Z}}^1(W_F^0, \hat{G})_{\ov{\bZ}[D^{-1}]}$ coming from restriction to the maximal ``pro-$D$" subgroup of the inertia subgroup $I_F$. Because of topological issues alluded to above, this does not literally make sense, but nonetheless we construct in Section~\ref{section:tame-reduction} a decomposition
\begin{align}\label{align:main-decomp}
\underline{\rm{Z}}^1(W_F^0, \hat{G})_{\ov{\bZ}[D^{-1}]} = \coprod_{\substack{\phi_D \in \Phi_D \\ \alpha_D \in \Sigma(\phi_D)}} \underline{\rm{Z}}^1(W_F^0, \hat{G})_{\phi_D, \alpha_D}
\end{align}
where, informally, $\Phi_D$ is a set of cocycles from the maximal pro-$D$ subgroup of $I_F$ to $\hat{G}$ and $\Sigma(\phi_D)$ comes from the failure of the centralizer of $\prescript{L}{}{\phi_D}$ in $\hat{G}$ to be connected. The following theorem, which is slightly generalized in Theorem~\ref{theorem:main-general}, is our main result.

\begin{theorem}\label{theorem:intro-MAIN-main}
    Every summand in (\ref{align:main-decomp}) is connected.
\end{theorem}

In \cite[\S 4.5]{Dat-notes}, one finds detailed conjectures on the blocks of $\Rep_{\ov{\bZ}[1/p]}(G(F))$ suggested by Conjecture~\ref{conj:4.3}. In particular, the components of $\underline{\rm{Z}}^1(W_F^0, \hat{G})_{\ov{\bZ}[1/p]}$ are expected to be in one-to-one correspondence with the so-called ``stable blocks" of $\Rep_{\ov{\bZ}[1/p]}(G(F))$, and there are precise expectations about the decomposition of each stable block into blocks. These conjectures have been partially confirmed for the subcategory of $\ov{\bZ}[1/p]$-representations of depth $0$ in \cite{Dat-Lanard}, following the works \cite{Lanard1}, \cite{Lanard2} which prove similar results for $\ov{\bZ}_{\ell}$-representations of depth $0$. It seems reasonable to expect that the conjectures of \cite[\S 4.5]{Dat-notes} can be extended to $\Rep_{\ov{\bZ}[D^{-1}]}(G(F))$ for any $D$. However, there are subtleties even when $R = \ov{\bZ}_{\ell}$, as noted in \cite[\S 4.7]{Dat-notes}. Thus we leave precise conjectures to the reader.\smallskip

Our proof of Theorem~\ref{theorem:intro-MAIN-main} is similar in several respects to the proof sketched in \cite[\S 5.4]{Dat-notes} in the case $R = \ov{\bZ}[1/p]$ and the action of $I_F$ preserves a pinning, but the details are somewhat different. In fact, to remain as self-contained as possible (and since they do not significantly simplify the argument), we do not rely on any previously-established connectedness results. Let us outline the proof in the case that $R = \ov{\bZ}[1/p]$ and the action of $W_F$ is inner, assumptions which allow us to avoid some technical complications while retaining the key features of the problem.
\begin{enumerate}
    \item (Section~\ref{section:tame-reduction}) Reduce to proving connectedness of $\underline{\rm{Z}}^1(W_F^0/P_F, \hat{G})_{\ov{\bZ}[1/p]}$ when the $W_F$-action on $\hat{G}$ is tame and preserves a Borel pair. (Since $R = \ov{\bZ}[1/p]$ in this outline, this reduction was already made in \cite{DHKM}.)
    \item (Section~\ref{ss:sc}) Analyze the behavior of $\underline{\rm{Z}}^1(W_F^0/P_F, \hat{G})_{\ov{\bZ}[1/p]}$ under certain homomorphisms $\hat{G} \to \hat{H}$ to reduce to the case that $\hat{G}$ is semisimple and simply connected.
    \item (Section~\ref{ss:fibral}) Define a map $\ov{\Sigma}\colon \underline{\rm{Z}}^1(W_F^0/P_F, \hat{G}) \to \hat{G}/\!/\hat{G}$ which, roughly, records the eigenvalues of a generator of tame inertia. Using ``purity" of $\underline{\rm{Z}}^1(W_F^0/P_F, \hat{G})$, reduce to the claims that $(\im \ov{\Sigma})_{\ov{\bZ}[1/p]}$ is connected, and for every field $k$ the fibers of $\ov{\Sigma}_k$ are connected.
    \item (Section~\ref{ss:qs}) Using the theory of complete reducibility and theorems of Steinberg, show that $\ov{\Sigma}^{-1}(x)$ is connected for every field $k$ and every $x \in (\im \ov{\Sigma})(k)$.
    \item (Section~\ref{ss:eigenvalues}, Appendix~\ref{appendix:a}) Prove that $(\im \ov{\Sigma})_{\ov{\bZ}[1/p]}$ is connected.
\end{enumerate}

Step (5) is the most interesting one. In Section~\ref{ss:fibral}, we write down connected closed subschemes $C_w$ of $(\im \ov{\Sigma})_{\ov{\bZ}[1/p]}$ indexed by elements $w$ of the Weyl group $W_0$ such that $(\im \ov{\Sigma})_{\ov{\bZ}[1/p]} = \bigcup_{w \in W_0} C_w$. This shows $\pi_0((\im \ov{\Sigma})_{\ov{\bZ}[1/p]}) = W_0/{\sim}$, where $\sim$ is the equivalence relation generated by the relations $w \sim w'$ if $C_w \cap C_{w'} \neq \emptyset$. If there is a field $k$ of characteristic $\neq p$ such that the action map $W_F \to \hat{G}^{\rm{ad}}_k$ lifts to a map $W_F \to \hat{G}_k$, then we exhibit a single element which lies in every $C_w(k)$, thus proving (5). In fact, this assumption is satisfied whenever the center $Z(\hat{G})$ is of order divisible by at most one prime $\neq p$, which holds in all simple types other than A. In type A, this argument can fail: Example~\ref{example:type-A} shows that when $\hat{G} = \SL_6$ and $p = 7$, there are two elements $w, w' \in W_0$ such that $C_w \cap C_{w'} = \emptyset$. This leads to more work in type A, which is handled in Appendix~\ref{appendix:a}. \smallskip

To aid the reader, steps (4) and (5) are described in the special case $\hat{G} = \SL_2$ in Section~\ref{section:example}. Section~\ref{section:prelim} contains recollections from the literature, as well as a few basic lemmas and a proof of the twisted Chevalley--Steinberg isomorphism over a general base scheme.

\subsection{Notation}\label{ss:intro-notation}

Throughout this paper, let $F$ be a nonarchimedean local field with residue field of order $q$, a power of the prime $p$. Let $W_F$ be the Weil group of $F$, let $K$ be a number field, let $D$ be a saturated multiplicatively closed subset of $\bZ$, and let $\hat{G}$ be a split reductive $\cO_K[D^{-1}]$-group scheme (in the sense of \cite[XIX, 2.7]{SGA3III}, i.e., with connected fibers) equipped with an action of $W_F$ which factors through a finite quotient $W$ and preserves a Borel pair $(\hat{B}, \hat{T})$. The definitions and notation surrounding $W_F^0$ and $\underline{\rm{Z}}^1(W_F^0, \hat{G})$ are given in Section~\ref{section:example} and retained past that section. Throughout, if $R$ is a ring, let $D_R$ denote the set of integers which are invertible in $R$. If $\fp$ is a prime ideal of $R$, let $k(\fp)$ denote the residue field of $R$ at $\fp$. If $H$ is a reductive $R$-group scheme, then $Z(H)$ denotes the center of $H$. If $\sigma$ is an automorphism of $H$, then we will use $\prescript{\sigma}{}{h}$ to denote the image of an element $h \in H(R)$ under $\sigma$.

\subsection{Acknowledgements}

I thank Robert Cass, Stephen DeBacker, and Alex Hazeltine for conversations which helped to clarify the ideas of this paper. I also thank the attendees of the Harish-Chandra seminar at the University of Michigan for allowing me to give a series of talks on this work and for asking many pertinent questions.

\section{An example}\label{section:example}

We begin by sketching the proof of Theorem~\ref{theorem:intro-MAIN} in the case $\hat{G} = \SL_2$, along the lines that we will use to prove the general case. This section is not logically necessary for the rest of the paper; we include it only in the hope that it will help orient the reader for the general argument. There are several features of $\SL_2$ which simplify the argument in this case, among which are that the group of outer automorphisms is trivial, the order of the center is only divisible by one prime, and the exponential map is defined on nilpotent matrices in arbitrary characteristic. \smallskip

For completeness, we briefly recall the construction of $\underline{\rm{Z}}^1(W_F^0, \hat{G})$ from \cite{DHKM}. Naive approaches run into thorny topological issues, and a major insight is that it is useful to introduce the subgroup $W_F^0 \subset W_F$ with \textit{discretized} inertia. Explicitly, there is a short exact sequence
\[
1 \to P_F \to W_F \to W_F/P_F \to 1
\]
where $P_F$ is the subgroup of wild inertia. Let $\rm{Fr}$ and $s$ be lifts of Frobenius and a topological generator of tame inertia, respectively, and let $(W_F/P_F)^0$ be the discrete subgroup of $W_F/P_F$ generated by $\rm{Fr}$ and $s$. We define $W_F^0$ to be the preimage of $(W_F/P_F)^0$ in $W_F$, and we topologize $W_F^0$ by declaring that $P_F$ is open with its usual profinite topology.\smallskip

We retain the notation of Section~\ref{ss:intro-notation}. Let $\underline{\rm{Z}}^1(W_F^0, \hat{G})$ be the moduli space of 1-cocycles $W_F^0 \to \hat{G}(R)$ with open kernel, where $R$ ranges over $\cO_K[1/p]$-algebras. By \cite[4.1]{DHKM}, this is a locally finite type syntomic $\cO_K[1/p]$-scheme which, when restricted to $\bZ_\ell \otimes \cO_K$-algebras for any $\ell \neq p$, parameterizes $\ell$-adically continuous L-parameters from $W_F$.\smallskip

By the reduction to tame parameters of \cite[\S 4.1]{DHKM}, it is fine to assume that the action of $W_F$ on $\SL_2$ is tame and only consider $\underline{\rm{Z}}^1(W_F^0/P_F, \SL_2)$. %We will assume, as we may, that $(\hat{B}, \hat{T})$ is the standard Borel pair of $\SL_2$.
Because the group of outer automorphisms of $\SL_2$ is trivial, $\rm{Fr}$ and $s$ act by $\Ad(g)$ and $\Ad(t)$ for some $g, t \in \hat{T}(\cO_K[1/p])$. These are required to satisfy $gtg^{-1} = t^q z$ for some central element $z \in Z(\SL_2)(\cO_K[1/p])$. Moreover, if $q$ is even then we may and do pass from $t$ to $-1 \cdot t$ if necessary to assume $z = 1$. \smallskip

Let $\ov{\Sigma}\colon \underline{\rm{Z}}^1(W_F^0/P_F, \SL_2) \to \bA^1$ be the map given by sending a cocycle $\vp$ to $\tr(\vp(s)t)$. The strategy of the proof of Theorem~\ref{theorem:intro-MAIN} is to show that $\ov{\Sigma}$ has connected fibers and that $(\im \ov{\Sigma})_{\ov{\bZ}[1/p]}$ is connected, and then use simple point-set topology to show that $\underline{\rm{Z}}^1(W_F^0/P_F, \SL_2)$ is connected. We will punt the point-set topology arguments (which are precisely the same for $\SL_2$ as for all other groups) to the coming sections and only handle properties of $\ov{\Sigma}$ in this section.

\begin{claim}\label{claim:conn-fibers}
    $\ov{\Sigma}$ has geometrically connected fibers.
\end{claim}

\begin{proof}
    Let $k$ be an algebraically closed field, and let $x \in (\im \ov{\Sigma})(k)$ be a point. By definition, there is some 1-cocycle $\vp\colon W_F^0/P_F \to \hat{G}(k)$ such that $\vp(s)t$ has eigenvalues $\alpha, \alpha^{-1}$. Let $\vp(s)t = t'(1+N)$ be the Jordan decomposition of $\vp(s)t$, where $N$ is nilpotent, and define a map $\vp_{-}\colon \bA_k^1 \to \ov{\Sigma}^{-1}(x)$ functorially via $\vp_c(\rm{Fr}) = \vp(\rm{Fr})$ and $\vp_c(s) = t'(1 + cN)t^{-1}$ for $c \in \bA_k^1$. The fact that $\vp_c$ is a cocycle for each $c$ can be checked directly using the defining relation $\rm{Fr} s \rm{Fr}^{-1} = s^q$. Note that $\vp_1 = \vp$ and that $\vp_0(s)t$ is semisimple.\smallskip

    Define $\Sigma\colon \underline{\rm{Z}}^1(W_F^0/P_F, \SL_2) \to \SL_2$ by $\Sigma(\vp) = \vp(s)t$, and note that $\tr^{-1}(x)$ contains a unique semisimple conjugacy class $C$: if $x = \alpha + \alpha^{-1}$ for some $\alpha \in k^\times$, then this is the conjugacy class of $\diag(\alpha, \alpha^{-1})$. By the previous paragraph, it suffices to show that $\Sigma^{-1}(C)$ is connected, and indeed it is enough to show that $\Sigma^{-1}(\diag(\alpha, \alpha^{-1}))$ is connected. However, using the defining relation for a cocycle, one checks that $\Sigma^{-1}(\diag(\alpha, \alpha^{-1}))$ is isomorphic to the centralizer of $\diag(\alpha^q, \alpha^{-q})$, which is either isomorphic to $\bGm$ or $\SL_2$ and in either case is connected.
\end{proof}

\begin{claim}\label{claim:conn-image}
    $(\im \ov{\Sigma})_{\ov{\bZ}[1/p]}$ is connected.
\end{claim}

\begin{proof}
    If $k$ is a field and $\vp$ is a cocycle $W_F^0/P_F \to \SL_2(k)$, then the relation $\Lvp(\rm{Fr})\Lvp(s)\Lvp(\rm{Fr})^{-1} = \Lvp(s)^q$ means $\prescript{\vp(\rm{Fr})g}{}{(\vp(s)t)}z^{-1} = (\vp(s)t)^q$. This restricts the possible eigenvalues of $\vp(s)t$ considerably: to see this, let $\{\alpha, \beta\}$ be the set of eigenvalues of $\vp(s)t$. If $z = 1$, then $\{\alpha, \beta\} = \{\alpha^q, \beta^q\}$, and hence $\alpha, \beta \in \mu_{q - 1}(k) \cup \mu_{q + 1}(k)$. If $z = -1$, then $\{\alpha, \beta\} = \{-\alpha^q, -\beta^q\}$, so $\alpha, \beta \in \mu_{2(q-1)}(k) \cup \mu_{2(q+1)}(k)$. Moreover, if $k$ is not of characteristic $2$, then $\alpha, \beta \not\in \mu_{q-1}(k) \cup \mu_{q+1}(k)$. These are, however, the only restrictions; for any $\alpha$ as above, it is straightforward to show that there is an L-parameter $\vp\colon W_F^0/P_F \to \SL_2(k)$ such that $\vp(s) = \diag(\alpha, \alpha^{-1})$. \smallskip

    If $z = 1$, then by the above paragraph, $\ov{\Sigma}$ factors (set-theoretically, at least) through a surjective map to the schematic image $Y_1$ of the map $\mu_{q-1} \cup \mu_{q+1} \to \bA^1$ given by $\alpha \mapsto \alpha + \alpha^{-1}$. The scheme $(Y_1)_{\ov{\bZ}[1/p]}$ is connected: indeed, since $\mu_{q-1}$ and $\mu_{q+1}$ meet at $1$, it suffices to note that any multiplicative type $\ov{\bZ}[1/p]$-group scheme which is finite of order prime to $p$ is connected. Next, suppose $z = -1$, so $q$ is odd. If $n \geq 1$, let $\mu'_{2n}$ be the nontrivial translate of the subgroup scheme $\mu_n$ in $\mu_{2n}$ over $\bZ[1/p]$. Again, the previous paragraph shows that $\ov{\Sigma}$ factors (set-theoretically) through a surjective map to the schematic image $Y_{-1}$ of the map $\mu'_{2(q-1)} \cup \mu'_{2(q+1)} \to \bA^1$ given by $\alpha \mapsto \alpha + \alpha^{-1}$. Over $\ov{\bQ}$ or $\ov{\bF}_{\ell}$ for $\ell \neq 2$, the schemes $\mu'_{2(q-1)}$ and $\mu'_{2(q+1)}$ do not meet, but they \textit{do} meet over $\ov{\bF}_2$ at the element $1$. Moreover, $\mu'_{2(q-1)}$ and $\mu'_{2(q+1)}$ are both connected. Thus we conclude as before that $Y_{-1}$ is connected.
\end{proof}

\section{Preliminaries}\label{section:prelim}

\subsection{Pure schemes}\label{ss:pure}

We will use the notion of \textit{purity} for a scheme morphism $X \to S$ from \cite[3.3.3]{Raynaud-Gruson}. The precise definition is not important in this paper, but we summarize some useful properties in the following lemma.

\begin{lemma}\label{lemma:purity-results}
    Let $R$ be a ring, and let $X$ be a flat finitely presented $R$-scheme.
    \begin{enumerate}
        \item If $X$ is $R$-pure and $R'$ is an $R$-algebra, then $X_{R'}$ is $R'$-pure.
        \item If $R$ is noetherian, then $X$ is $R$-pure if and only if, for every $R$-algebra $A$ which is a DVR, the scheme $X_A$ is $A$-pure.
        \item If $R$ is a field, then $X$ is $R$-pure.
        \item If $R$ is a complete DVR with uniformizer $\pi$ and $X$ is affine, then $X$ is $R$-pure if and only if the coordinate ring $R[X]$ is $\pi$-adically separated.
        \item If $R$ is a DVR with fraction field $K$ and residue field $k$ and $X$ is $R$-pure, then the Zariski closure of every irreducible component of $X_K$ meets $X_k$.
    \end{enumerate}
\end{lemma}

\begin{proof}
    (1) is proved in \cite[3.3.7]{Raynaud-Gruson}. (3) and (5) follow immediately from the definition in \cite[3.3.3]{Raynaud-Gruson}. For (4), note that $R[X]$ is countably generated, so by \cite[Thm.\ 12]{Kaplansky} it is a direct sum of rank $1$ $R$-modules. In particular, $R[X]$ is $\pi$-adically separated if and only if it is $R$-free, so the result follows from \cite[3.3.5]{Raynaud-Gruson}. Finally, (2) is proved in \cite[2.2]{Hom-schemes}.
\end{proof}

The main point of purity for our purposes is that moduli spaces of L-parameters are pure.

\begin{lemma}\label{lemma:cocycles-pure}
    Let $K$ be a number field, and let $H$ be a split reductive group scheme over $\cO_K[1/p]$ equipped with a finite action of $W_F$. The scheme $\underline{\rm{Z}}^1(W_F^0, H)$ is flat and $\cO_K[1/p]$-pure.
\end{lemma}

\begin{proof}
    By Lemma~\ref{lemma:purity-results}(1), (2), and (3), we need only show that $\underline{\rm{Z}}^1(W_F^0, H)_{\bZ_\ell \otimes \cO_K}$ is $\bZ_\ell \otimes \cO_K$-pure for every prime number $\ell \neq p$. Every connected component of $\underline{\rm{Z}}^1(W_F^0, H)_{\bZ_\ell \otimes \cO_K}$ is flat, finite type, and affine by \cite[4.1]{DHKM}, and the same reference shows that the coordinate ring of each component is $\ell$-adically separated. (Strictly speaking, \cite[4.1]{DHKM} only applies as written when $K = \bQ$, but the proof never uses this assumption.) Thus Lemma~\ref{lemma:purity-results}(4) shows the result.
\end{proof}

To make reductions later, we need two elementary results about purity which are mainly topological in nature.

\begin{lemma}\label{lemma:conn-passes-to-image}
    Let $R$ be an integral domain, let $\fp \in \Spec R$, and let $X$ and $Y$ be finitely presented $R$-schemes such that $Y$ is flat and $R$-pure. Let $f\colon X \to Y$ be an $R$-morphism, and suppose
    \begin{enumerate}
        \item $f_{k(\fp)}\colon X_{k(\fp)} \to Y_{k(\fp)}$ is surjective,
        \item $X$ is connected.
    \end{enumerate}
    Then $Y$ is connected.
\end{lemma}

\begin{proof}
    By standard limit arguments, we may and do assume $R$ is noetherian. Since $Y$ is flat, each component of $Y$ has nonempty $\Frac(R)$-fiber. By purity and Lemma~\ref{lemma:purity-results}(5), it follows that each component has nonempty $k(\fp)$-fiber. Thus it suffices to show that any two closed points $y_1$, $y_2$ of $Y_{k(\fp)}$ are contained in the same component of $Y$. By (1), there exist points $x_1$, $x_2$ of $X_{k(\fp)}$ lifting $y_1$, $y_2$. By (2), there is a sequence $x_1 = s_1, \eta_1, s_2, \dots, \eta_n, s_n = x_2$ in $X$ such that $s_{i+1}$ specializes $\eta_i$ and $\eta_{i+1}$ for all $i$. The sequence $y_1 = f(s_1), f(\eta_1), \dots, f(s_n) = y_2$ provides a chain in $Y$ which shows that $y_1$ and $y_2$ lie in the same component.
\end{proof}

%\begin{lemma}\label{lemma:point-set-topology}
%    Let $X$ be a topological space, and let $\{Z_i\}_{i \in I}$ be a finite collection of connected closed subspaces of $X$. Let $\sim$ be the equivalence relation on $I$ generated by the relations $i \sim j$ if $Z_i \cap Z_j \neq \emptyset$. Then $\pi_0(X) = I/{\sim}$.
%\end{lemma}

%\begin{proof}
%    If $J \subset I$ is an equivalence class for $\sim$, then $Z_J \coloneqq \bigcup_{j \in J} Z_j$ is closed and connected. Moreover, $Z_J \cap Z_{I - J} = \emptyset$, so $Z_J$ is also open, hence a component of $X$.
%\end{proof}

\begin{lemma}\label{lemma:pure-to-finite-connected}
    Let $R$ be an integral domain, and let $f\colon X \to Y$ be a surjective morphism of finitely presented $R$-schemes such that $X$ is $R$-pure, $Y$ is $R$-finite, and $f^{-1}(y)$ is connected for all $y \in Y$. Then the map $\pi_0(f)\colon \pi_0(X) \to \pi_0(Y)$ is bijective.
\end{lemma}

The main point making this lemma nontrival is that $f$ is not necessarily flat; the conclusion does not hold if, for example, $Y$ is not assumed finite or $X$ is not assumed pure.

\begin{proof}
    It is enough to consider the case of noetherian $R$. First, $\pi_0(f)$ is clearly surjective. To see injectivity, let $C$ be a connected component of $Y$. If $K = \Frac(R)$, let $C_K = \{x_1, \dots, x_n\}$, and note that $\ov{\{x_i\}} \to \Spec R$ is surjective by finiteness. Each $\ov{\{x_i\}}$ is connected, and since $C$ is connected, we have $C = \bigcup_{i=1}^n \ov{\{x_i\}}$. Thus it suffices to show that $f^{-1}(\ov{\{x_i\}})$ is connected for all $i$.\smallskip
    
    By finiteness of $Y$, the fiber $f^{-1}(x_i)$ is a connected component of $X_K$ for all $i$. If $X_i$ denotes the Zariski closure of $f^{-1}(x_i)$ in $X$, then $X_i$ is connected for all $i$. We have
    \[
    f^{-1}(\ov{\{x_i\}}) = X_i \cup \bigcup_{\fq \in \Spec R} f^{-1}(\ov{\{x_i\}}_{k(\fq)}).
    \]
    By hypothesis, each $f^{-1}(\ov{\{x_i\}}_{k(\fq)})$ is connected, and by Lemma~\ref{lemma:purity-results}(5), the fiber $(X_i)_{k(\fq)}$ is nonempty. Thus $(X_i)_{k(\fq)}$ meets $f^{-1}(\ov{\{x_i\}}_{k(\fq)})$, and it follows that $f^{-1}(\ov{\{x_i\}})$ is connected.
\end{proof}

\subsection{Complete reducibility}\label{ss:cr}

Recall \cite[\S 2.1]{CGP} that if $H$ is a reductive group over a field $k$ and $\lambda\colon\bGm \to H$ is a cocharacter, then there are subgroups $P_H(\lambda), Z_H(\lambda) \subset H$ defined by
\[
P_H(\lambda) = \{h \in H\colon \, \lim_{t \to 0} \lambda(t)h\lambda(t)^{-1} \text{ exists}\} \text{ and }
Z_H(\lambda) = \{h \in H\colon \, \Ad(h)\lambda = \lambda\}.
\]
If $H$ is connected, then $P_H(\lambda)$ is a parabolic subgroup of $H$ and $Z_H(\lambda)$ is a Levi of $P_H(\lambda)$.\smallskip

If $\Gamma$ is a group, then a homomorphism $f\colon \Gamma \to H(k)$ is \textit{completely reducible} if for every cocharacter $\lambda\colon \bGm \to H$ such that $f$ factors through $P_H(\lambda)$, there is some cocharacter $\mu\colon \bGm \to H$ such that $P_H(\lambda) = P_H(\mu)$ and $f$ factors through $Z_H(\mu)$. The following three results are well-known.

\begin{theorem}\cite[3.7]{BMR05}\label{theorem:closed-orbit}
    If $\Gamma$ is finitely generated, then $f\colon \Gamma \to H(k)$ is completely reducible if and only if the $H$-orbit of $f$ in $\uHom_{k\textrm{-}\rm{gp}}(\Gamma, H)$ is closed.
\end{theorem}

\begin{proof}
    A surjection from a finitely generated free group to $\Gamma$ induces an $H$-equivariant closed embedding $\uHom_{k\textrm{-}\rm{gp}}(\Gamma, H) \to H^n$, so the result follows from \cite[3.7]{BMR05}.
\end{proof}

\begin{theorem}\cite[3.10]{BMR05}\label{theorem:bmr-normal}
    If $f\colon \Gamma \to H(k)$ is completely reducible and $N \subset \Gamma$ is a normal subgroup, then $f|_N$ is also completely reducible.
\end{theorem}

We recall that an automorphism $\sigma$ of $H^0$ is \textit{quasi-semisimple} if it preserves a Borel pair of $H^0$. 

\begin{lemma}\label{lemma:bz-semisimple}
    If $f\colon \bZ \to H(k)$ is a homomorphism, then $f$ is completely reducible if and only if the conjugation action of $f(1)$ on $H^0$ is quasi-semisimple.
\end{lemma}

\begin{proof}
    This follows from \cite[Prop.\ 1]{Springer-twisted}.
\end{proof}

\subsection{A twisted Chevalley--Steinberg isomorphism}\label{ss:chevalley-steinberg}

Let $R$ be a ring, and let $H$ be a reductive $R$-group scheme. Let $\sigma$ be an $R$-automorphism of $H$ which preserves a Borel pair $(B, T)$. There is an action of $H$ on itself defined by $h \cdot_{\sigma} h' \coloneqq h \cdot h' \cdot \prescript{\sigma}{}{h}^{-1}$, which we call \textit{$\sigma$-twisted conjugacy}. Denote the resulting GIT quotient by $H/\!/_{\sigma} H$, and let $\chi_{\sigma} \colon H \to H/\!/_{\sigma} H$ be the natural map. Equivalently, one can view $H/\!/_{\sigma} H$ as the quotient $(H \ltimes \sigma)/\!/ H$, where $H$ acts by conjugation, and one can consider $\chi_{\sigma}$ as the natural map $H \ltimes \sigma \to (H \ltimes \sigma)/\!/ H$; this is the perspective we will take in Section~\ref{ss:fibral}. \smallskip

Let $W_0$ be the subgroup of $\sigma$-fixed points of the Weyl group of $T$ in $H$, and let $A = T/(1 - \sigma)T$. Let $W_0$ act on $A$ as follows: if $w_0 \in W_0$ and $n_0$ is a lift of $w_0$ in $N_{H}(T)$ (valued in some fppf $R$-algebra), then define $w_0 \cdot t = n_0 t \left(\prescript{\sigma}{}{n_0}\right)^{-1}$. Restricting $\chi_{\sigma}$ to $T$ induces a map $j\colon A/W_0 \to H/\!/_\sigma H$.\smallskip

In \cite[Thm.\ 1]{Springer-twisted} and \cite[4.2.3]{Xiao-Zhu}, it is shown that $j$ is an isomorphism if $R$ is a field; in \cite[4.1]{Lee-adjoint}, this is shown if $\sigma$ is the identity; in (the proof of) \cite[6.6]{DHKM}, this is shown provided that $R$ is a Dedekind domain. The aim of this section is to explain a formal way to extend these results to general $R$, akin to the arguments of \cite{Lee-adjoint}.

\begin{prop}\label{prop:twisted-chevalley-split}
    If $H$ is a split reductive $\bZ$-group scheme and $\sigma$ is a pinning-preserving automorphism of $H$, then the natural map $j\colon A/W_0 \to H/\!/_{\sigma} H$ is an isomorphism.
\end{prop}

\begin{proof}
    The argument of \cite[6.6]{DHKM} can be read verbatim with the tuple $(H, B, T, \sigma, W_0)$ in place of $(C_{\hat{G}}(\phi)^0, B_{\phi}, T_{\phi}, \beta, (\Omega_{\phi}^\circ)^{\Ad \beta})$, and $\bZ$ in place of the base ring $\cO_{\widetilde{K}_e}[\frac{1}{pN_{\hat{G}}}]$ (with notation as in \textit{loc.\ cit.}).
\end{proof}

The following lemma is a straightforward generalization of (the proof of) \cite[4.2]{Lee-adjoint}, and presumably well-known, but we give an argument for want of a reference.

\begin{lemma}\label{lemma:enough-to-check-GIT-on-primes}
    Let $R$ be a noetherian ring, let $H$ be a flat affine $R$-group scheme, and let $X$ be a flat affine $R$-scheme with an action of $H$. The formation of the GIT quotient $X/\!/H \coloneqq \Spec R[X]^H$ commutes with arbitrary base change if and only if, for every prime ideal $\fp \in \Spec R$, the base change map $X/\!/H \otimes_R k(\fp) \to X_{k(\fp)}/\!/H_{k(\fp)}$ is an isomorphism.
\end{lemma}

\begin{proof}
    Suppose $X = \Spec C$, and let $\alpha\colon C \to R[H] \otimes_R C$ be the homomorphism corresponding to the action map, so $C^H = \ker(\alpha - 1)$. For each $\fp \in \Spec R$, observe the commutative diagram
    \[
    \begin{tikzcd}
        C^H \otimes_R k(\fp) \arrow[r] \arrow[d]
            &C \otimes_R k(\fp) \arrow[r] \arrow[d]
            &\im(\alpha - 1) \otimes_R k(\fp) \arrow[r] \arrow[d]
            &0 \\
        (C \otimes_R k(\fp))^{H_{k(\fp)}} \arrow[r]
            &C \otimes_R k(\fp) \arrow[r]
            &\im(\alpha_{k(\fp)} - 1) \arrow[r]
            &0
    \end{tikzcd}
    \]
    which shows, by our hypothesis and the five lemma, that the map $\im(\alpha - 1) \otimes_R k(\fp) \to \im(\alpha_{k(\fp)} - 1)$ is an isomorphism. Moreover, if $i\colon \im(\alpha - 1) \to R[H] \otimes_R C$ is the natural inclusion, then there is a commutative diagram
    \[
    \begin{tikzcd}
        \im(\alpha - 1) \otimes_R k(\fp) \arrow[r] \arrow[d]
            &R[H] \otimes_R C \otimes_R k(\fp) \arrow[r] \arrow[d]
            &\coker(i) \otimes_R k(\fp) \arrow[r] \arrow[d]
            &0 \\
        \im(\alpha_{k(\fp)} - 1) \arrow[r]
            &(k(\fp))[H] \otimes_{k(\fp)} C_{k(\fp)} \arrow[r]
            &\coker(i_{k(\fp)}) \arrow[r]
            &0
    \end{tikzcd}
    \]
    which shows, again by the five lemma, that $\coker(i) \otimes_R k(\fp) \to \coker(i_{k(\fp)})$ is an isomorphism. It follows that $i \otimes_R k(\fp)$ is injective, and hence $\Tor_1^R(\coker(i), k(\fp)) = 0$. Thus $\coker(i)$ is $R$-flat, so the natural map $C^H \otimes_R R' \to (C \otimes_R R')^{H_{R'}}$ is an isomorphism for all $R'$, as desired.
\end{proof}

\begin{theorem}\label{theorem:twisted-chevalley}
    Let $X$ be a scheme, let $H$ be a reductive $X$-group scheme admitting a Borel pair $(B, T)$, and let $\sigma$ be an $X$-automorphism of $H$ preserving $(B, T)$. The natural map $j\colon A/W_0 \to H/\!/_\sigma H$ is an isomorphism. The formation of each of $A/W_0$ and $H/\!/_\sigma H$ commutes with arbitrary base change.
\end{theorem}

\begin{proof}
    Since the formation of GIT quotients commutes with flat base change, we may pass to fppf covers of $X$ at will. In particular, by \cite[XXII, 2.3]{SGA3III}, we may assume that $X = \Spec R$ is affine and $G$ is split, and by \cite[XXIV, 1.3]{SGA3III} we may write $\sigma = \Ad(t) \circ \sigma_1$, where $\sigma_1$ preserves the pinning and $t \in T(R)$. The map $H \to H$ given by $h \mapsto ht$ is an isomorphism which intertwines $\sigma$-twisted conjugacy and $\sigma_1$-twisted conjugacy and thus induces an isomorphism $H/\!/_{\sigma} H \to H/\!/_{\sigma_1} H$. There is a similar isomorphism for $A/W_0$, and thus we may and do pass from $\sigma$ to $\sigma_1$ to assume that $\sigma$ preserves a pinning.\smallskip

    By \cite[XXV, 1.1]{SGA3III}, there is a tuple $(\bH, \bB, \bT, \boldsymbol{\sigma})$ consisting of a split reductive $\bZ$-group scheme $\bH$, a Borel subgroup scheme $\bB \subset \bH$, a maximal torus $\bT \subset \bB$, and a pinning-preserving automorphism $\boldsymbol{\sigma}$ of $\bH$ such that $(\bH, \bB, \bT, \boldsymbol{\sigma}) \otimes R \cong (H, B, T, \sigma)$. Let $\bA$ denote the group scheme of $\boldsymbol{\sigma}$-coinvariants of $\bT$. For every ring $R'$, there is a commutative diagram
    \[
    \begin{tikzcd}
        \bZ[\bG]^{\bG, \boldsymbol{\sigma}} \otimes R' \arrow[r] \arrow[d]
            &\bZ[\bA]^{W_0} \otimes R' \arrow[d] \\
        R'[\bG_{R'}]^{\bG_{R'}, \boldsymbol{\sigma}_{R'}} \arrow[r]
            &R'[\bA_{R'}]^{W_0}.
    \end{tikzcd}
    \]
    By Proposition~\ref{prop:twisted-chevalley-split}, the top and right arrows are isomorphisms, so to prove the theorem it is enough to show that the bottom arrow is an isomorphism for all $R'$. By Lemma~\ref{lemma:enough-to-check-GIT-on-primes}, we need only check this when $R' = \bZ/p$ for some prime number $p$. This case is \cite[Thm.\ 1]{Springer-twisted}.
\end{proof}

\section{Reduction to tame parameters}\label{section:tame-reduction}

In this section, we will describe the decomposition (\ref{align:main-decomp}) promised in the introduction, and we will show that each summand in this decomposition is equivalent to a scheme of ``tame" cocycles. We retain the notation of the introduction, letting $R$ be a $\ov{\bZ}[1/p]$-algebra and letting $D$ be the set of integers which are invertible in $R$. Let $D'$ be the multiplicatively closed subset of $\bZ$ generated by the prime numbers not in $D$. For the basic definitions associated with $\underline{\rm{Z}}^1(W_F^0, \hat{G})$, see Section~\ref{section:example}.

\subsection{Motivation}\label{ss:motivation}

Our construction may initially appear strange, so we begin with two special cases, both of which appear in \cite{DHKM}. Consider first the case $D = \{\pm p^n\colon n \geq 0\}$, which will give rise to the decomposition (\ref{eqn:z-1/p-decomp}). There is a restriction map $r_P\colon \underline{\rm{Z}}^1(W_F^0, \hat{G}) \to \underline{\rm{Z}}^1(P_F, \hat{G})$, and it is shown in \cite[3.1]{DHKM} that this is a discrete invariant: in other words, there is a decomposition
\[
\underline{\rm{Z}}^1(W_F^0, \hat{G})_{\ov{\bZ}[1/p]} = \coprod_{\phi \in \Phi} \hat{G} \cdot r_P^{-1}(\phi)
\]
for some set of sections $\Phi \subset \underline{\rm{Z}}^1(P_F, \hat{G})(\ov{\bZ}[1/p])$. Given $\phi \in \Phi$, there is a closed subgroup scheme
\[
Z_{\prescript{L}{}{G}}(\phi) = \{(g, w) \in \LG \colon (g, w)\prescript{L}{}{\phi}(w^{-1}pw)(g, w)^{-1} = \prescript{L}{}{\phi}(p) \,\forall p \in P_F\},
\]
and if $\vp$ extends $\phi$ then $\vp$ factors through $Z_{\prescript{L}{}{G}}(\phi)$. This gives rise to an induced homomorphism $\alpha_{\vp}\colon W_F^0 \to \pi_0(Z_{\prescript{L}{}{G}}(\phi))$. Again $\alpha_{\vp}$ is a discrete invariant of $r_P^{-1}(\phi)$: if $\Sigma(\phi)$ is the set of such $\alpha_\vp$, then there is a decomposition
\[
    \underline{\rm{Z}}^1(W_F^0, \hat{G})_{\ov{\bZ}[1/p]} = \coprod_{\substack{\phi \in \Phi \\ \alpha \in \Sigma(\phi)}} \hat{G} \cdot r_P^{-1}(\phi)_\alpha,
\]
where $r_P^{-1}(\phi)_{\alpha}$ is the clopen subscheme of $r_P^{-1}(\phi)$ consisting of those cocycles $\vp$ such that $\alpha_{\vp} = \alpha$. This is (\ref{eqn:z-1/p-decomp}). The content of ``tame reduction" in this case is that each $r_P^{-1}(\phi)_\alpha$ is isomorphic to a scheme of tame cocycles $\underline{\rm{Z}}^1(W_F^0/P_F, Z_{\LG}(\phi)^0)_{\ov{\bZ}[1/p]}$, where $Z_{\LG}(\phi)^0$ is equipped with a finite $W_F$-action preserving a Borel pair.\footnote{Crucially, even if the $W_F$-action on $\hat{G}$ is defined over $\bZ[1/p]$ and preserves a pinning, we do not know whether we can ensure the same for the $W_F$-action on $Z_{\LG}(\phi)^0$. This justifies working in the generality we do.} This reduces one to understanding spaces of \textit{tame} cocycles.\smallskip

Next consider the other extreme case $D = \{n\colon \ell \nmid n\}$, where $\ell$ is a fixed prime, and let us assume for simplicity that $R = \ov{\bF}_{\ell}$; the case $R = \ov{\bZ}_{\ell}$ is similar. Now \cite[4.1]{DHKM} shows that if $A$ is an $\ov{\bF}_\ell$-algebra, then $\underline{\rm{Z}}^1(W_F^0, \hat{G})_{\ov{\bF}_{\ell}}(A)$ parameterizes \textit{continuous} $1$-cocycles $W_F \to \hat{G}(A)$, where $\hat{G}(A)$ is equipped with the discrete topology. Thus there is a natural restriction map $r^{\ell}\colon \underline{\rm{Z}}^1(W_F^0, \hat{G})_{\ov{\bF}_{\ell}} \to \underline{\rm{Z}}^1_{\rm{cts}}(I_F^{\ell}, \hat{G})_{\ov{\bF}_{\ell}}$, where $I_F^{\ell}$ denotes the maximal closed subgroup of $I_F$ of pro-order prime to $\ell$. Using a similar argument to the above, one shows that there is a decomposition
\[
\underline{\rm{Z}}^1(W_F^0, \hat{G})_{\ov{\bF}_{\ell}} = \coprod_{\substack{\phi^\ell \in \Phi^\ell \\ \alpha^\ell \in \Sigma(\phi^\ell)}} \hat{G}_{\ov{\bF}_\ell} \cdot (r^\ell)^{-1}(\phi^\ell)_{\alpha^\ell},
\]
where $\Phi^\ell$ is a set of continuous cocycles $I_F^\ell \to \hat{G}(\ov{\bF}_\ell)$ and $\Sigma(\phi^{\ell})$ is a set of component group maps $W_F^0 \to \pi_0(Z_{\LG}(\phi^{\ell})$. Here the content of ``tame reduction" is that for each nonempty $(r^\ell)^{-1}(\phi^\ell)_{\alpha^\ell}$, there is an isomorphism $(r^\ell)^{-1}(\phi^\ell)_{\alpha^\ell} \cong \underline{\rm{Z}}^1(W_F/I_F^\ell, Z_{\LG}(\phi^\ell)^0)_{\ov{\bF}_{\ell}}$. Again, this reduces one to studying spaces of cocycles which one might call ``$\ell'$-tame". \smallskip

For more general $R$, to obtain a similar decomposition, we will consider a reduction map $r_D$ which, informally, records the restriction of a cocycle $\vp$ to the maximal pro-$D$ subgroup of $I_F$. To see how to formulate this precisely, let us reframe the previous case: for simplicity, assume that $W_F$ acts trivially on $\hat{G}$ and consider a continuous homomorphism $f\colon W_F/P_F \to \hat{G}(\ov{\bQ}_\ell)$. If $f(s) = tu$ is the Jordan decomposition in $\hat{G}(\ov{\bQ}_\ell)$, then $t$ is of some finite order $n = \ell^c d$, where $\ell \nmid d$. The restriction of $f$ to $I_F/P_F$ is determined by $c$ and $t^{\ell^c}$, so there is a restriction map
\[
\Hom(W_F, \hat{G}(\ov{\bQ}_\ell)) \to \varinjlim_{c, d} \Hom(\langle s^{\ell^c} \rangle/\langle s^{\ell^c d} \rangle, \hat{G}(\ov{\bQ}_\ell))
\]
given by sending $f$ to the map $s^{\ell^c m} \mapsto t^{\ell^c m}$, under which the image of $f$ determines $f|_{I_F^{\ell}}$.\smallskip

Because each quotient $\langle s^{\ell^c} \rangle/\langle s^{\ell^c d} \rangle$ is finite, the above discussion suggests an algebraic way to formulate a restriction map $r^\ell$, which can be defined without reference to the profinite group $W_F$. In this form, the generalization to $r_D$ is simple, although a modicum of care is needed since the Jordan decomposition need not exist over general rings. After defining $r_D$, our version of tame reduction reduces one to studying spaces of ``$D$-tame" cocycles.

\subsection{Constructions}

We now begin by defining $r_D$. Let $\sU_{\hat{G}}$ be the unipotent scheme of $\hat{G}$, as defined in \cite[\S 4]{Cotner-Springer} to be the schematic closure of the unipotent variety of $\hat{G}_{\bQ}$. We note that this definition is not completely standard when $\hat{G}$ is not simply connected, but it is determined by the conditions that $\sU_{\hat{G}}$ is $\bZ$-flat and $\sU_{\hat{G}}(k)$ is the set of unipotent elements of $\hat{G}(k)$ for every field $k$. We warn the reader that the fibers of $\sU_{\hat{G}}$ over $\bZ$ are not generally reduced.

\begin{lemma}\label{lemma:power-iso-unipotent-scheme}
    If $N \in D$, then the $N$th power map $[N]\colon (\sU_{\hat{G}})_{\bZ[D^{-1}]} \to (\sU_{\hat{G}})_{\bZ[D^{-1}]}$ is a $\hat{G}$-equivariant isomorphism.
\end{lemma}

\begin{proof}
    Equivariance is clear. By the fibral isomorphism criterion \cite[IV\textsubscript{4}, 17.9.5]{EGA}, it suffices to show that for every field $k$ over $\bZ[D^{-1}]$, the map $[N]_k$ is an isomorphism. If $\chara k = 0$, then this is clear from the existence of the exponential map. If $\chara k = p > 0$, then there is some $n$ such that $[p^n]_k$ kills $(\sU_{\hat{G}})_k$: to see this, we can reduce to $\hat{G} = \GL_n$, in which case $\sU_{\hat{G}}$ is reduced and we reduce to the fact that $\sU_{\GL_n}(k)$ is killed by $p^n$. If $M$ is such that $NM \equiv 1 \pmod{p^n}$, then $[NM]_k$ is the identity on $\sU_{\hat{G}}$ and thus $[N]_k$ is an isomorphism, as desired.
\end{proof}

\begin{lemma}\label{lemma:bounded-order-of-inertia}
    Let $P$ be an open subgroup of $P_F$ which is normal in $W_F^0$. There is an integer $n = n_P$ such that $s^n$ acts trivially on $P_F/P$ and for every $\bZ[1/p]$-algebra $R$ and every cocycle $\vp\colon W_F^0 \to \hat{G}(R)$, the power $\Lvp(s)^n$ lies in $\sU_{\hat{G}}(R)$.
\end{lemma}

\begin{proof}
    By \cite[2.2]{DHKM}, there is some $n$ such that for every field $k$ over $\bZ[1/p]$ and every $\vp\colon W_F^0 \to \hat{G}(k)$, the power $\Lvp(s)^n$ lies in $\sU_{\hat{G}}(k)$. Thus, since $\underline{\rm{Z}}^1(W_F^0, \hat{G})_{\bQ}$ is reduced \cite[4.1]{DHKM}, the map $\pi_n\colon \underline{\rm{Z}}^1(W_F^0, \hat{G})_{\bZ[1/p]} \to \hat{G}$ given by $\vp \mapsto \Lvp(s)^n$ factors through $\sU_{\hat{G}}$ generically. By Lemma~\ref{lemma:cocycles-pure}, the source is flat and thus $\pi_n$ factors through $\sU_{\hat{G}}$. Since $P_F/P$ is finite, we may pass to a multiple of $n$ to assume $s^n$ acts trivially on $P_F/P$.
\end{proof}

If $P$ is an open subgroup of $P_F$ which is normal in $W_F^0$ and $n \geq 1$ is such that $s^n$ acts trivially on $P_F/P$, let $I_{F, n, P}^0$ denote the subgroup of $I_F^0$ generated by $s^n$ and $P$. Let $I_{F, n}^0 = I_{F, n, P_F}^0$. If $n$ is as in Lemma~\ref{lemma:bounded-order-of-inertia}, write $n = ab$ for $a \in D$ and $b \in D'$. Define a $\bZ[D^{-1}]$-morphism 
\[
r_{n, P, D}\colon \underline{\rm{Z}}^1(W_F^0/P, \hat{G})_{\bZ[D^{-1}]} \to \underline{\rm{Z}}^1(I_{F, b}^0/I_{F, n, P}^0, \hat{G})_{\bZ[D^{-1}]}
\]
as follows: if $R$ is a $\bZ[D^{-1}]$-algebra and $\vp\colon W_F^0/P \to \hat{G}(R)$ is a cocycle, note that $\Lvp(s)^{ab}$ is unipotent, so by Lemma~\ref{lemma:power-iso-unipotent-scheme} there is a unique element $u \in \sU_{\hat{G}}(R)$ whose $a$th power is $\Lvp(s)^n$. Since $[a]$ is a $\hat{G}$-equivariant automorphism of $\sU_{\hat{G}}$, it follows that $u$ commutes with $\Lvp(p)$ for all $p \in P_F$ and with $\Lvp(s)^b$. Thus we may define $\vp_{P, D} = r_{n, P, D}(\vp)$ via $\vp_{n, P, D}|_{P_F} = \vp|_{P_F}$ and $\Lvp_{P, D}(s^b) = u^{-1} \cdot \Lvp(s)^b$.\smallskip

Since $I_{F, b}^0/I_{F, n, P}^0$ is a finite group of order in $D$ for any $n$, \cite[A.9]{DHKM} shows that there is a set $\Phi_{P, n} \subset \rm{Z}^1(I_{F, b}^0/I_{F, n, P}^0, \hat{G}(\ov{\bZ}[D^{-1}]))$ such that
\[
\underline{\rm{Z}}^1(I_{F, b}^0/I_{F, n, P}^0, \hat{G})_{\ov{\bZ}[D^{-1}]} = \coprod_{\phi_n \in \Phi_{P, n}} \hat{G}_{\ov{\bZ}[D^{-1}]} \cdot \phi_n,
\]
where $\hat{G}_{\ov{\bZ}[D^{-1}]} \cdot \phi_n$ denotes the orbit of $\phi_n$, isomorphic to $\hat{G}_{\ov{\bZ}[D^{-1}]}/Z_{\hat{G}_{\ov{\bZ}[D^{-1}]}}(\phi_n)$. In particular, if $P' \subset P$ and $n \mid n' = a'b'$, where $n' = n_{P'}$ as in Lemma~\ref{lemma:bounded-order-of-inertia}, then the natural map $\underline{\rm{Z}}^1(I_{F, b}^0/I_{F, n, P}^0, \hat{G})_{\ov{\bZ}[D^{-1}]} \to \underline{\rm{Z}}^1(I_{F, b'}^0/I_{F, n', P'}^0, \hat{G})_{\ov{\bZ}[D^{-1}]}$ is a clopen embedding, and the maps $r_{n, P, D}$ and $r_{n', P', D}$ are compatible in the obvious sense. Therefore we may define
\[
r_D \coloneqq \varinjlim_{n, P} r_{n, P, D} \colon \underline{\rm{Z}}^1(W_F^0, \hat{G})_{\bZ[D^{-1}]} \to X_D \coloneqq \varinjlim_{n, P} \underline{\rm{Z}}^1(I_{F, b}^0/I_{F, n, P}^0, \hat{G})_{\bZ[D^{-1}]}
\]
because $\underline{\rm{Z}}^1(W_F^0, \hat{G}) = \varinjlim_P \underline{\rm{Z}}^1(W_F^0/P, \hat{G})$.\smallskip

If $R$ is a $\bZ[D^{-1}]$-algebra and $\phi_D \in X_D(R)$, we define
\[
Z_{\LG_R}(\phi_D) = \{(g, w) \in \LG_R\colon (g, w)\cdot \prescript{L}{}{\phi}_D \cdot (g, w)^{-1} = \prescript{L}{}{\phi}_D \circ \Ad(w)\}
\]
and $Z_{\hat{G}_R}(\phi_D) = Z_{\LG_R}(\phi_D) \cap \hat{G}_R$. A key observation is that, if $\vp\colon W_F^0 \to \hat{G}(R)$ is a cocycle such that $\vp(I_F^0)$ is finite, then $\Lvp$ factors through $Z_{\LG_R}(r_D(\vp))$.

\begin{theorem}\label{theorem:DHKM-3.1}
    There is a set $\Phi_D \subset X_D(\ov{\bZ}[D^{-1}])$ such that
    \[
    X_D = \coprod_{\phi_D \in \Phi_D} \hat{G} \cdot \phi_D
    \]
    For each $\phi_D \in \Phi_D$, the centralizer $Z_{\hat{G}_{\ov{\bZ}[D^{-1}]}}(\phi_D)$ is a smooth affine $\ov{\bZ}[D^{-1}]$-group scheme with split reductive identity component and constant component group.
\end{theorem}

\begin{proof}
    The points not already observed above follow from \cite[A.12, A.13]{DHKM} as in \cite[3.1]{DHKM}; the key point is that, if $\phi_D$ is represented by a cocycle $\phi_n\colon I_{F, b}^0/I_{F, n, P}^0 \to \hat{G}(\ov{\bZ}[D^{-1}])$, then $Z_{\LG}(\phi_D) = Z_{\LG}(\phi_n)$.
\end{proof}

By Theorem~\ref{theorem:DHKM-3.1}, for every $\phi_D \in \Phi_D$ there is an exact sequence
\[
1 \to \pi_0(Z_{\hat{G}}(\phi_D)) \to \widetilde{\pi}_0(\phi_D) \coloneqq \pi_0(Z_{\LG}(\phi_D)) \to W.
\]
If $R$ is a $\ov{\bZ}[D^{-1}]$-algebra and $\vp\colon W_F^0 \to \hat{G}(R)$ restricts to $\phi_D$, then there is an induced homomorphism $\alpha_{\vp}\colon W_F^0 \to \widetilde{\pi}_0(\phi_D)$, and we see that the image of $\alpha_{\vp}$ in $\varinjlim_{n, P} \Hom(I_{F, b}^0/I_{F,n,P}, \widetilde{\pi}_0(\phi_D))$ agrees with the element induced by $\phi_D$. Moreover, the composition $W_F^0 \to \widetilde{\pi}_0(\phi_D) \to W$ is the natural projection map. Let $\Sigma(\phi_D)$ be the set of homomorphisms $\alpha_D$ satisfying these two conditions. By Theorem~\ref{theorem:DHKM-3.1}, there is a disjoint union decomposition
\[
\underline{\rm{Z}}^1(W_F^0, \hat{G})_{\ov{\bZ}[D^{-1}]} = \coprod_{\substack{
    \phi_D \in \Phi_D \\
    \alpha_D \in \Sigma(\phi_D)}} \hat{G} \times^{Z_{\hat{G}}(\phi_D)} r_D^{-1}(\phi_D)_{\alpha_D},
\]
where by definition $r_D^{-1}(\phi_D)_{\alpha_D}$ is the closed subscheme of $\underline{\rm{Z}}^1(W_F^0, \hat{G})_{\ov{\bZ}[D^{-1}]}$ consisting of those cocycles $\vp$ with $r_D(\vp) = \phi_D$ and $\alpha_{\vp} = \alpha_D$. We call the pair $(\phi_D, \alpha_D)$ \textit{admissible} if $\underline{\rm{Z}}^1(W_F^0, \hat{G})_{\phi_D, \alpha_D} \coloneqq \hat{G} \times^{Z_{\hat{G}}(\phi_D)} r_D^{-1}(\phi_D)_{\alpha_D}$ is nonempty. Note that, in view of this decomposition and \cite[4.1]{DHKM}, each $\underline{\rm{Z}}^1(W_F^0, \hat{G})_{\phi_D, \alpha_D}$ is syntomic.\smallskip

We can now state the main theorem of this paper, which was stated imprecisely and in a slightly weaker form in Theorem~\ref{theorem:intro-MAIN-main}.

\begin{theorem}\label{theorem:main-general}
    Let $\phi_D \in \Phi_D$ and $\alpha_D \in \Sigma(\phi_D)$ be admissible, and let $R$ be an integral domain over $\ov{\bZ}[1/p]$ with $D_R = D$. Then $\underline{\rm{Z}}^1(W_F^0, \hat{G})_{\phi_D, \alpha_D} \otimes_{\ov{\bZ}[D^{-1}]} R$ is connected.
\end{theorem}

Note that the statement of Theorem~\ref{theorem:main-general} is general enough that it applies to $R = \ov{\bZ}_{\ell}$ and $R = \ov{\bF}_{\ell}$ for any prime $\ell \neq p$, so it recovers \cite[4.8]{DHKM} as a special case. The first step which allows us to analyze $r_D^{-1}(\phi_D)_{\alpha_D}$ is a ``reduction to tame parameters" as in \cite{DHKM}, the main input for which is provided by the following theorem.

\begin{theorem}\label{theorem:tame-reduction}
    If $\phi_D \in \Phi_D$ and $\alpha_D \in \Sigma(\phi_D)$ are admissible, then there is a $1$-cocycle $\vp\colon W_F^0 \to \hat{G}(\ov{\bZ}[D^{-1}])$ with $r_D(\vp) = \phi_D$ and $\alpha_{\vp} = \alpha_D$ such that $\Lvp(W_F^0)$ is finite and $\Ad_{\Lvp}$ preserves a Borel pair of $Z_{\hat{G}}(\phi_D)^0$.
\end{theorem}

\begin{proof}
    Let $(B, T)$ be a Borel pair of $Z_{\hat{G}}(\phi_D)^0$ which is defined over $\ov{\bZ}[D^{-1}]$. Let $\cT$ be the normalizer of $(B, T)$ in $Z_{\hat{G}}(\vp_D)$, and note that $\cT^0 = T$ and $\cT$ has component group $\widetilde{\pi}_0(\phi_D)$. Thus $\alpha_D$ gives rise to a natural action of $W_F$ on $T$. Let $P \subset P_F$ be open and normal in $W_F^0$ such that $\phi_D|_P$ is trivial, and let $n = ab$ be an integer as in Lemma~\ref{lemma:bounded-order-of-inertia}, where $a \in D$ and $b \in D'$. Thus $\phi_D$ is realized as a $1$-cocycle $I_{F, b}^0/I_{F, n, P}^0 \to \hat{G}(\ov{\bZ}[D^{-1}])$. We may form the scheme $\underline{\rm{Z}}^1_{\alpha_D}(W_F^0/I_{F, b}^0, T)$ of $1$-cocycles, which is a diagonalizable group scheme by the same argument as \cite[3.6(1)]{DHKM}. Similarly, let $\underline{\Sigma}(W_F^0, \cT)_{\phi_D, \alpha_D}$ be the scheme of those $\vp \in \underline{\rm{Z}}^1(W_F^0, \hat{G})_{\phi_D, \alpha_D}$ such that $\Lvp(W_F^0) \subset \cT$. It is straightforward to check that $\underline{\Sigma}(W_F^0, \cT)_{\phi_D, \alpha_D}$ is a pseudo-torsor for $\underline{\rm{Z}}^1_{\alpha_D}(W_F^0/I_{F, b}^0, T)$ in the sense of \cite[\href{https://stacks.math.columbia.edu/tag/0497}{Tag 0497}]{stacks-project}, so by the discussion following \cite[3.6]{DHKM} it suffices to show that there is a faithfully flat $\ov{\bZ}[D^{-1}]$-algebra $R$ such that $\underline{\Sigma}(W_F^0, \cT)_{\phi_D, \alpha_D}$ admits an $R$-point with finite image.\smallskip

    If $\underline{\rm{Z}}^1(W_F^0, \hat{G})_{\phi_D, \alpha_D}$ is nonempty, then by flatness it contains a $\ov{\bQ}$-point $\vp$. The last paragraph of the section \textit{Lifting this point to characteristic $0$} in the proof of \cite[3.6]{DHKM} shows that there exists such a $\vp$ with finite image, and by passing to a conjugate we may and do assume $\Ad_{\Lvp}$ preserves $(B, T)$. Using the fact that $\ov{\bQ} = \varinjlim_N \ov{\bZ}[1/N]$, we may and do assume that $\vp$ factors through $\hat{G}(\ov{\bZ}[D^{-1}, 1/N])$ for some $N$. Let $K$ be a number field such that $\Lvp$ factors through $\hat{G}(\cO_K[D^{-1}, 1/N])$, and let $\lambda$ be a prime of $\cO_K$ not dividing any element of $D$; it is enough to show now that, after possibly extending $K$, the L-homomorphism $\Lvp$ can be conjugated to factor through $\hat{G}(\cO_{K_{\lambda}})$. Since $\Lvp$ has finite (hence bounded) image, this follows from \cite[1.3]{Martin-Vinberg}.
\end{proof}

In general, if $R$ is a $\ov{\bZ}[D^{-1}]$-algebra, let $\underline{\rm{Z}}^1(W_F^0/P_F, \hat{G})_R^{r_D = 1}$ denote the clopen subscheme of $\underline{\rm{Z}}^1(W_F^0/P_F, \hat{G})_R$ consisting of those $\vp$ satisfying $r_D(\vp) = 1$. For given admissible $\phi_D$ and $\alpha_D$, choose $\vp\colon W_F^0 \to \hat{G}(\ov{\bZ}[D^{-1}])$ as in Theorem~\ref{theorem:tame-reduction}. The map $\rho \mapsto \rho \cdot \vp$ gives an isomorphism
\[
\underline{\rm{Z}}^1_{\vp}(W_F^0/P_F, Z_{\hat{G}}(\phi_D)^0)_R^{r_D = 1} \xrightarrow[]{\sim} r_D^{-1}(\phi_D)_{\alpha_D} \otimes_{\ov{\bZ}[D^{-1}]} R.
\]
Note that the action of $W_F$ on $Z_{\hat{G}}(\phi_D)^0$ induced by $\vp$ is \textit{$D$-tame} in the following sense.

\begin{definition}\label{def:d-tame}
    The action of $W_F$ on $\hat{G}$ is \textit{$D$-tame} if it is trivial on $I_{F, b}^0$ for some $b \in D'$.
\end{definition}

Thus Theorem~\ref{theorem:intro-MAIN-main} reduces to the following a priori weaker statement.

\begin{theorem}\label{theorem:main-general-reduced}
    Suppose that the action of $W_F$ on $\hat{G}$ is $D$-tame and preserves a Borel pair. If $R$ is an integral domain over $\ov{\bZ}[1/p]$ and $D_R = D$, then $\underline{\rm{Z}}^1(W_F^0/P_F, \hat{G})^{r_D = 1}_R$ is connected.
\end{theorem}

\section{Proof of Theorem~\ref{theorem:intro-MAIN-main}}\label{section:proof}

For the rest of this paper, we keep the notation of Section~\ref{ss:intro-notation} and assume that the action of $W_F$ on $\hat{G}$ is $D$-pure in the sense of Definition~\ref{def:d-tame}. Moreover, let $R$ be an integral domain over $\ov{\bZ}[1/p]$ such that $D_R = D$, and let $D'$ be as in Section~\ref{section:tame-reduction}. As we have shown in the previous section, to prove Theorem~\ref{theorem:intro-MAIN-main} (or the slightly stronger Theorem~\ref{theorem:main-general}), it is enough to prove Theorem~\ref{theorem:main-general-reduced}.

\subsection{Reduction to the simply connected case}\label{ss:sc}

The goal of this section is to reduce Theorem~\ref{theorem:main-general-reduced} to the case that $\hat{G}$ is semisimple and simply connected. In fact, for completeness we prove slightly more than is necessary. We begin by establishing Theorem~\ref{theorem:main-general-reduced} in the case that $\hat{G}$ is a torus; the case $R = \ov{\bZ}[1/p]$ is established in \cite[3.14]{Dat-notes} by a completely similar argument.

\begin{lemma}\label{lemma:torus-case}
    Let $\hat{T}$ be a torus over $R$ equipped with a finite $D$-tame action of $W_F^0/P_F$. The scheme $\underline{\rm{Z}}^1(W_F^0/P_F, \hat{T})^{r_D = 1}_R$ is connected.
\end{lemma}

\begin{proof}
    Let $n$ be an integer as in Lemma~\ref{lemma:bounded-order-of-inertia}, where $P = P_F$, and let $b$ be the largest factor of $n$ lying in $D'$. Thus $\underline{\rm{Z}}^1(W_F^0/P_F, \hat{T})^{r_D = 1}_R$ is the kernel of the map $\hat{T}_R \times \hat{T}_R \to \hat{T}_R \times \hat{T}_R$ given by
    \[
    (\Phi_0, \Sigma_0) \mapsto (\Phi_0 \cdot \prescript{\rm{Fr}}{}{\Sigma_0} \cdot \prescript{s^q}{}{\Phi_0}^{-1} \cdot \left(\prod_{i=0}^{q-1} \prescript{s^i}{}{\Sigma_0}\right)^{-1}, \prod_{i=0}^{b-1} \prescript{s^i}{}{\Sigma_0}).
    \]
    By \cite[IX, 6.8]{SGA3II}, this is a diagonalizable $R$-group scheme, so it is enough to show that its character group has torsion order in $D'$. If the action of $s$ is trivial, then this is clear because $\underline{\rm{Z}}^1(W_F^0/P_F, \hat{T})^{r_D = 1}_R$ is a closed $R$-subgroup scheme of $\hat{T} \times \hat{T}[b]$ which surjects onto the first factor of $\hat{T}$. If instead the map $L_s\colon t \mapsto t (\prescript{s}{}{t})^{-1}$ is an isogeny, then the map $m_b\colon t \mapsto \prod_{i=0}^{b-1} \prescript{s^i}{}{t}$ is trivial. Indeed, since the action is $D$-tame, $s^b$ acts trivially on $\hat{T}$, so the homomorphism $m_b\colon \hat{T} \to \hat{T}$ factors through the finite $\ker(L_s)$, hence is trivial. Thus the map $\underline{\rm{Z}}^1(W_F^0/P_F, \hat{T})^{r_D = 1}_R \to \hat{T}_R$ given by $\vp \mapsto \vp(s)$ is an isogeny with kernel $\ker(L_{s^q})$, which is of order in $D'$ by \cite[1.2(1)]{Digne-Michel-quass}.\smallskip

    In general, the isogeny $\hat{T} \to \hat{T}_s \times \hat{T}/\hat{T}^s$ is of degree in $D'$ by \cite[1.2(3)]{Digne-Michel-quass}. Thus the map
    \[
    \underline{\rm{Z}}^1(W_F^0/P_F, \hat{T})^{r_D = 1}_R \to \underline{\rm{Z}}^1(W_F^0/P_F, \hat{T}_s)^{r_D = 1}_R \times \underline{\rm{Z}}^1(W_F^0/P_F, \hat{T}/\hat{T}^s)^{r_D = 1}_R
    \]
    has finite kernel and cokernel of order in $D'$ and the lemma follows.
\end{proof}

\begin{lemma}\label{lemma:pushforward-surj}
    Let $\ell$ be prime, and let $f\colon H \to H'$ be an isogeny of connected reductive $\ov{\bF}_\ell$-groups which is equivariant with respect to finite actions of a finitely generated group $\Gamma$. If $f$ is of $\ell$-power degree, then the pushforward $f_*\colon \underline{\rm{Z}}^1(\Gamma, H) \to \underline{\rm{Z}}^1(\Gamma, H')$ is a homeomorphism.
\end{lemma}

\begin{proof}
    Since $\Gamma$ is finitely generated, its actions on $H$ and $H'$ are both defined over some finite field $\bF_{\ell^n}$. Note that $\ker f$ is finite and connected, so, after possibly increasing $n$, there is an isogeny $h\colon H' \to H$ such that $g \circ f$ is equal to the relative Frobenius $F\colon H \to H^{(\ell^n)} \cong H$. Note that $g$ is automatically $\Gamma$-equivariant because $f$ and $F$ both are. Moreover, the pushforward $F_*$ is equal to the relative Frobenius of $\underline{\rm{Z}}^1(\Gamma, H)$: one can see this by embedding $\underline{\rm{Z}}^1(\Gamma, H)$ into $H^n$ via a surjection from a free group to $\Gamma$. Thus $F_*$ is a homeomorphism, and a symmetric argument with $g$ in place of $f$ shows that $f_*$ is a homeomorphism.
\end{proof}

%Now suppose that $\Gamma = W_F^0/P_F$ and $f$ is of $p$-power degree. In this case, it suffices to show that every cocycle $\vp'\colon W_F^0/P_F \to H'(\ov{\bF}_{\ell})$ lifts to a cocycle $\vp\colon W_F^0/P_F \to H(\ov{\bF}_{\ell})$. By \cite[4.1]{DHKM}, the cocycle $\vp'$ extends continuously to $W_F/P_F$. The obstruction to lifting $\vp'$ to $\vp$ lies in the group $\rm{H}^2(W_F/P_F, \ker f)$, which vanishes by \cite[3.8]{DHKM}.

\begin{lemma}\label{lemma:reduction-sssc}
    Let $f\colon \hat{G}' \to \hat{G}$ be a homomorphism of reductive $\cO_K[D^{-1}]$-group schemes which is equivariant with respect to finite $D$-tame actions of $W_F/P_F$. If $f$ induces an isogeny of derived groups, then $f_*\colon \underline{\rm{Z}}^1(W_F^0/P_F, \hat{G}')^{r_D = 1}_R \to \underline{\rm{Z}}^1(W_F^0/P_F, \hat{G})^{r_D = 1}_R$ is surjective on $\pi_0$.
\end{lemma}

\begin{proof}
    First suppose that $f$ is an isogeny. Let $n$ be the degree of $f$, and factor $n = \prod_{i=1}^m \ell_i^{k_i}$, where $\ell_i \neq \ell_j$ for $i \neq j$. There is a unique factorization
    \[
    f = \left( \hat{G}' = \hat{G}_m \xrightarrow[]{f_m} \hat{G}_{m-1} \xrightarrow[]{f_{m-1}} \cdots \xrightarrow[]{f_1} \hat{G}_0 = \hat{G} \right),
    \]
    where $f_i$ is of degree $\ell_i^{k_i}$. Note that the action of $W_F$ on $\hat{G}$ extends uniquely to actions on $\hat{G}_i$ for all $i$ by canonicity, and each $f_i$ is $W_F$-equivariant. \smallskip
    
    For each $1 \leq i \leq m$, there is a pushforward $(f_i)_*\colon\underline{\rm{Z}}^1(W_F^0/P_F, \hat{G}_i)^{r_D = 1} \to \underline{\rm{Z}}^1(W_F^0/P_F, \hat{G}_{i - 1})^{r_D = 1}$. If $\ell_i \in D'$, the map $((f_i)_*)_{\cO_K/\ell_i}$ is a homeomorphism by Lemma~\ref{lemma:pushforward-surj}, and hence $(f_i)_*$ is surjective on $\pi_0$ by Lemma~\ref{lemma:cocycles-pure}.\smallskip
    
    Next suppose $\ell_i \in D$. We claim that if $k$ is an algebraically closed field over $R$, then $((f_i)_*)_k$ is surjective. We first reduce to the case that $k$ is of positive characteristic. Note that by standard limit arguments, there is a saturated multiplicatively closed subset $D_1 \subset D$ such that $D_1$ contains only finitely many primes, the action of $W_F$ on $\hat{G}$ is defined over $\cO_K[D_1^{-1}]$, and $\underline{\rm{Z}}^1(W_F^0/P_F, \hat{G})^{r_{D_1} = 1}_R = \underline{\rm{Z}}^1(W_F^0/P_F, \hat{G})^{r_D = 1}_R$. If $\chara k = 0$, then to show that $((f_i)_*)_k$ is surjective it follows from a standard constructibility argument that it is enough to show that $((f_i)_*)_{\cO_K/\ell'}$ is surjective for all $\ell' \not\in D_1$. So suppose that $k$ is of characteristic $\ell'$.\smallskip
    
    By Lemma~\ref{lemma:bounded-order-of-inertia} and the definitions, any cocycle $\vp_{i-1} \in \underline{\rm{Z}}^1(W_F^0/P_F, \hat{G}_{i - 1})^{r_D = 1}(k)$ satisfies $\Lvp(s)^b = 1$ for some $b \in D'$. If $N_b$ is the smallest normal subgroup of $W_F/P_F$ containing $s^b$, then $\vp$ induces a cocycle $W_F/N_b \to \hat{G}_{i-1}(k)$ and the obstruction to lifting $\vp_{i-1}$ to $\underline{\rm{Z}}^1(W_F^0/P_F, \hat{G}_i)^{r_D = 1}(k)$ lies in $\rm{H}^2(W_F/N_b, \ker(f_i)(k))$. By the Hochschild--Serre spectral sequence, there is an exact sequence
    \[
    \rm{H}^1(W_F/I_F, \rm{H}^1(I_F/N_b, \ker(f_i)(k))) \to \rm{H}^2(W_F/N_b, \ker(f_i)(k)) \to \rm{H}^2(I_F/N_b, \ker(f_i)(k))^{W_F/I_F}
    \]
    Since $I_F/N_b$ is finite of order coprime to the order of $\ker(f_i)(k)$, we see that the outer terms of this exact sequence vanish and hence the middle term vanishes as well. So we see that $((f_i)_*)_k$ is surjective, and by the above it follows that $(f_i)_*$ is surjective. We have shown each $(f_i)_*$ is surjective on $\pi_0$ and thus $f_*$ is surjective on $\pi_0$.\smallskip

    Now let $Z$ (resp.\ $Z'$) be the maximal central torus of $\hat{G}$ (resp.\ $\hat{G}'$), and let $\sD(\hat{G})$ (resp.\ $\sD(\hat{G}')$) be its derived group. The multiplication morphism $m\colon Z \times \sD(\hat{G}) \to \hat{G}$ (resp.\ $m'\colon Z' \times \sD(\hat{G}') \to \hat{G}'$) is a central isogeny by \cite[XXII, 6.2.4]{SGA3III}. There is a commutative diagram
    \[
    \begin{tikzcd}
        Z' \times \sD(\hat{G}') \arrow[r, "\widetilde{f}"] \arrow[d, "m'"]
            &Z \times \sD(\hat{G}) \arrow[d, "m"] \\
        \hat{G}' \arrow[r, "f"]
            &\hat{G}
    \end{tikzcd}
    \]
    and we have shown above that $m'_*$ and $m_*$ are both surjective on $\pi_0$. Moreover, $\underline{\rm{Z}}^1(W_F^0/P_F, Z)$ and $\underline{\rm{Z}}^1(W_F^0/P_F, Z')$ are both connected by Lemma~\ref{lemma:torus-case}, and $\widetilde{f}\colon \sD(\hat{G}') \to \sD(\hat{G})$ is an isogeny, so $\widetilde{f}_*$ is also surjective on $\pi_0$. Thus $f_*$ is surjective on $\pi_0$, as desired.
\end{proof}

\subsection{A map to the adjoint quotient}\label{ss:fibral}

As in Section~\ref{ss:chevalley-steinberg}, denote the GIT quotient $(\hat{G} \rtimes s)/\!/\hat{G}$ by $\hat{G}/\!/_s \hat{G}$, and let $\chi_s\colon \hat{G} \rtimes s \to \hat{G}/\!/_s \hat{G}$ be the natural map. Let $W_0$ be the subgroup of $s$-fixed points of the Weyl group of $\hat{T}$ in $\hat{G}$, and let $A = \hat{T}/(1 - s)\hat{T}$. There is a natural map $\hat{T} \to \hat{G}/\!/_s \hat{G}$ given by $t \mapsto \chi_s(t \rtimes s)$, and Theorem~\ref{theorem:twisted-chevalley} shows that the induced map $A/W_0 \to \hat{G}/\!/_s \hat{G}$ is an isomorphism.\smallskip

Define an $\cO_K[D^{-1}]$-morphism $\Sigma\colon \underline{\rm{Z}}^1(W_F^0/P_F, \hat{G}) \to \hat{G} \rtimes s$ by $\vp \mapsto \prescript{L}{}{\vp}(s)$. Define $\ov{\Sigma} = \chi_s \circ \Sigma$.\smallskip

There are maps $\hat{G} \rtimes s \to \hat{G} \rtimes s^q$ given by the $q$th power map and $\rm{Fr}$, and these induce $\rm{Fr}^{-1}[q]\colon \hat{G}/\!/_s \hat{G} \to \hat{G}/\!/_s \hat{G}$. Let $B_{\hat{G}} = (\hat{G}/\!/_s \hat{G})^{\rm{Fr}^{-1}[q]}$. Moreover, let $B_{\hat{G}}[D']$ denote the closed subscheme of $(B_{\hat{G}})_{\cO_K[D^{-1}]}$ consisting of those sections $x$ satisfying $[b]x = \chi_s(1)$ for some $b \in D'$. Define $\Sigma_{D'}$ and $\ov{\Sigma}_{D'}$ to be the restrictions of $\Sigma$ and $\ov{\Sigma}$ to $\underline{\rm{Z}}^1(W_F^0/P_F, \hat{G})^{r_D = 1}$, and note that $\ov{\Sigma}_{D'}$ factors through $B_{\hat{G}}[D']$. In fact, $\underline{\rm{Z}}^1(W_F^0/P_F, \hat{G})^{r_D = 1} = \ov{\Sigma}^{-1}(B_{\hat{G}}[D'])$.\smallskip

For each $w \in W_0$, let $A_w$ be the subscheme of $A$ consisting of those $a$ such that $\left(\prescript{\rm{Fr}^{-1}}{}{a}\right)^q = na\left(\prescript{s}{}{n}\right)^{-1}$, where $n$ is a lift of $w$ in $N_{\hat{G}}(\hat{T})$. Let $A_w[D']$ denote the closed subscheme of $A_w$ consisting of those $a$ such that $\chi_s(a^b) = \chi_s(1)$ for some $b \in D'$.

\begin{lemma}\label{lemma:Aw-B-surjective}
    Each $(A_w[D'])_R$ is finite and connected, and the natural map $\bigcup_{w \in W_0} A_w[D']_R \to B_{\hat{G}}[D']_R$ is surjective.
\end{lemma}

\begin{proof}
    Note that $A_w[D']$ is a torsor for the multiplicative type $\cO_K[D^{-1}]$-group scheme $A_w'[D']$ consisting of those $a \in A$ such that $a^b = 1$ for some $b \in D'$ and $\left(\prescript{\rm{Fr}^{-1}}{}{a}\right)^q = \prescript{w}{}{a}$. Since each $A_w'$ is a multiplicative type group scheme, to show that $A_w$ is finite and connected it suffices to show that $A_w'(\ov{\bQ})$ contains no elements of order in $D$. But this is clear.\smallskip
    
    By Theorem~\ref{theorem:twisted-chevalley}, the natural map $(A/W_0)^{\rm{Fr}^{-1}[q]} \to B_{\hat{G}}$ is an isomorphism. Moreover, the map $\bigcup_{w \in W_0} A_w[D'] \to (A/W_0)^{\rm{Fr}^{-1}[q]}[D']$ is clearly surjective.
\end{proof}

As we will see in Section~\ref{ss:eigenvalues}, the map of Lemma~\ref{lemma:Aw-B-surjective} is the key to showing that $(\im \ov{\Sigma}_{D'})_R$ is connected and thereby (using the results of this section) to proving Theorem~\ref{theorem:main-general-reduced}.

\begin{lemma}\label{lemma:lifting-Aw}
    For every algebraically closed field $k$ of characteristic $\neq p$, every $w \in W_0$, and every $a \in A_w[D'](k)$, there is some $\vp \in \underline{\rm{Z}}^1(W_F^0/P_F, \hat{G}(k))$ such that $\ov{\Sigma}(\vp) = \chi_s(a)$.
\end{lemma}

\begin{proof}
    Note that $A_{\prescript{\rm{Fr}^{-1}}{}{w}}(k)$ is, equivalently, the set of $a \in A(k)$ such that $a^q = \prescript{w\rm{Fr}}{}{a} \cdot n \left(\prescript{s^q}{}{n}\right)^{-1}$. Lifting $a$ arbitrarily to $t_0 \in \hat{T}(k)$, we find that there is some $t' \in \hat{T}(k)$ such that
    \[
    \prod_{i=0}^{q-1} \prescript{s^i}{}{t_0} = \prescript{w\rm{Fr}}{}{t_0} \cdot n\left(\prescript{s^q}{}{n}\right)^{-1} \cdot t' \left(\prescript{s}{}{t'}\right)^{-1}.
    \]
    Now $(1 - s)\hat{T} = (1 - s^q)\hat{T}$: indeed, clearly $(1 - s^q)\hat{T} \subset (1 - s)\hat{T}$, so these are equal for dimension reasons because $(1 - s^q)\hat{T} = (\rm{Fr} - \rm{Fr}s)\rm{Fr}^{-1}\hat{T}$. Moreover, $(1 - s)\hat{T}$ is clearly $W_0$-stable. In particular, there is some $t'' \in \hat{T}(k)$ such that $\prescript{w}{}{\left(t''\left(\prescript{s^q}{}{t''}\right)^{-1}\right)} = t'\left(\prescript{s}{}{t'}\right)^{-1}$. We thus define $\vp \in \rm{Z}^1(W_F^0/P_F, \hat{G}(k))$ by $\vp(\rm{Fr}) = nt''$ and $\vp(s) = t_0$.
\end{proof}

\begin{lemma}\label{lemma:finite-image-git}
    The map $\ov{\Sigma}$ factors through a surjective map to $B_{\hat{G}}$. Moreover, the schematic image $\im \ov{\Sigma}$ is finite flat. Completely similar claims hold for $\ov{\Sigma}_{D'}$ and $\im \ov{\Sigma}_{D'}$ with $B_{\hat{G}}[D']$ in place of $B_{\hat{G}}$.
\end{lemma}

\begin{proof}
    Since $\underline{\rm{Z}}^1(W_F^0/P_F, \hat{G})$ is flat over the Dedekind domain $\cO_K[1/p]$, it is clear that $\im \ov{\Sigma}$ is also flat. The fact that $\ov{\Sigma}$ factors through $B_{\hat{G}}$ follows directly from the fact that
    \[
    \Lvp(\rm{Fr})\Lvp(s)\Lvp(\rm{Fr})^{-1} = \Lvp(s)^q
    \]
    for $\vp \in \underline{\rm{Z}}^1(W_F^0/P_F, \hat{G})$. Moreover, the set-theoretic image of $\ov{\Sigma}$ is equal to $B_{\hat{G}}$ by Lemma~\ref{lemma:Aw-B-surjective} and Lemma~\ref{lemma:lifting-Aw}. Finiteness of $B_{\hat{G}}$, and hence $\im \ov{\Sigma}$, thus follows from Lemma~\ref{lemma:Aw-B-surjective}. The surjectivity of $\ov{\Sigma}$ follows from Lemma~\ref{lemma:Aw-B-surjective} and Lemma~\ref{lemma:lifting-Aw}. The claims for $\ov{\Sigma}_{D'}$ are immediate consequences.
\end{proof}

\begin{lemma}\label{lemma:reduction-fixed-jordan}
    Suppose that for every field $k$ over $\cO_K[D^{-1}]$ and every $x \in (\im \ov{\Sigma}_{D'})(k)$, the preimage $\ov{\Sigma}^{-1}(x)$ is connected. Then the map $\pi_0(\underline{\rm{Z}}^1(W_F^0/P_F, \hat{G})^{r_D = 1}_R) \to \pi_0((\im \ov{\Sigma}_{D'})_R)$ is bijective.
\end{lemma}

\begin{proof}
    This follows from Lemma~\ref{lemma:finite-image-git}, Lemma~\ref{lemma:cocycles-pure}, and Lemma~\ref{lemma:pure-to-finite-connected}.
\end{proof}

%\begin{remark}
%    With notation as in the proof of Lemma~\ref{lemma:finite-image-git}, we note that it is \textit{not} true that $B_G = (\bigcup_{w \in W_0} A_w)/W_0$. For example, let $q = 3$, let $G = \SL_2$ with trivial action, and let $k$ be a field of characteristic $2$. Then one checks $B_G \cong \Spec k[x]/(x^3)$, while $(\bigcup_{w \in W_0} A_w)/W_0 = \mu_4/\rm{inversion} \cong \Spec k[u, v]/(u^2, v^2 - 1, uv - u)$.
%\end{remark}

\begin{remark}
    If $\hat{G}$ is semisimple and simply connected and the action of $W_F$ on $\hat{G}$ is unramified, then \cite[3.6]{Shotton-irr} and \cite[4.1]{DHKM} combine with fppf descent to show that $B_{\hat{G}}$ is flat with \'etale generic fiber, so in this case we have $\im \ov{\Sigma} = B_{\hat{G}}$. In fact, \cite[3.8]{Shotton-irr} suggests that this equality holds even for ramified actions. Since we do not need this refined result, we have not attempted to prove it. For $\hat{G}$ which are not simply connected, it is not clear that $B_{\hat{G}}$ is flat.
\end{remark}

\subsection{Deformation to quasi-semisimple inertia}\label{ss:qs}

In this section, we use the results of Section~\ref{ss:cr} to show that every L-homomorphism $\Lvp\colon W_F^0/P_F \to \LG(k)$ can be deformed to a cocycle such that $\Lvp(s)$ has quasi-semisimple action on $\hat{G}$, and use this to deduce a connectedness claim.

\begin{lemma}\label{lemma:reduce-to-ss}
    Let $k$ be a field, let $x \in (\hat{G}/\!/_s \hat{G})(k)$ be a point, and let $C \subset \chi_s^{-1}(x)$ be the unique closed $\hat{G}$-orbit. The inclusion $\Sigma^{-1}(C) \to \ov{\Sigma}^{-1}(x)$ induces a surjection on $\pi_0$.
\end{lemma}

\begin{proof}
    We may and do assume $k = \ov{k}$. Let $\Lvp \in \ov{\Sigma}^{-1}(x)(k)$. By Theorem~\ref{theorem:closed-orbit} and the Hilbert-Mumford criterion, there is a cocharacter $\lambda\colon \bGm \to \hat{G}$ such that $\prescript{L}{}{\psi} \coloneqq \lim_{t \to 0} \Ad(\lambda(t)) \circ \Lvp$ exists and is completely reducible. By Theorem~\ref{theorem:bmr-normal}, the restriction $\prescript{L}{}{\psi}|_{I_F^0/P_F}$ is completely reducible, so Lemma~\ref{lemma:bz-semisimple} shows that $\prescript{L}{}{\psi}(s) \in C(k)$, so we have produced a morphism $f\colon \bA^1 \to \ov{\Sigma}^{-1}(x)$ with $f(1) = \Lvp$ and $f(0) \in \Sigma^{-1}(C)(k)$.
\end{proof}

\begin{lemma}\label{lemma:connected-fibers}
    Let $k$ be a field over $\cO_K[D^{-1}]$, and let $x \in (\im \ov{\Sigma}_{D'})(k)$. If $\hat{G}$ is semisimple and $|\pi_1(\hat{G})| \in D$, then $\ov{\Sigma}^{-1}(x)$ is connected.
\end{lemma}

\begin{proof}
    By Lemma~\ref{lemma:reduce-to-ss}, it is enough to show that $X = \Sigma^{-1}(C)$ is connected. Note that $\Sigma$ restricts to a surjective $k$-morphism $X \to C$. There is a natural map $[q]\colon C \to \hat{G} \rtimes s^q$ given by the $q$th power, whose scheme-theoretic image we call $C^{[q]}$. Being a surjective morphism of homogeneous spaces for $\hat{G}$, the map $C \to C^{[q]}$ is flat. Let $H$ denote the fiber product
    \[
    \begin{tikzcd}
        H \arrow[r] \arrow[d]
            &\hat{G} \times C^{[q]} \arrow[d, "{(g, t) \mapsto (g t g^{-1},\, t)}"] \\
        C \arrow[r, "\Delta \circ {[q]}"]
            &C^{{[q]}} \times C^{{[q]}}
    \end{tikzcd}
    \]
    so $H$ is a ``universal twisted centralizer" over $C$, and the diagram shows that $H \to C$ is flat. Moreover, there is an action map $\alpha\colon H \times_C \Sigma^{-1}(C) \to \Sigma^{-1}(C)$ given by $\alpha(h, \vp)(\rm{Fr}) = h\vp(\rm{Fr})$ and $\alpha(h, \vp)(s) = \vp(s)$, which makes $\Sigma|_X$ into an $H$-torsor over $C$. By \cite[9.11]{Steinberg-End}, since $|\pi_1(\hat{G})| \in D$ and each $t \in C(k)$ is quasi-semisimple of (finite) order in $D'$, it follows that $H$ has connected fibers over $C$. Since $C$ is a single $\hat{G}$-orbit, it is in particular connected, and thus $X$ is connected.
\end{proof}

\subsection{Connectedness of the scheme of eigenvalues}\label{ss:eigenvalues}

In this section, we complete the proof of Theorem~\ref{theorem:main-general-reduced} by showing that $(\im \ov{\Sigma}_{D'})_R$ is connected. We retain the notation of Section~\ref{ss:fibral}.\smallskip

By Lemma~\ref{lemma:lifting-Aw} and Lemma~\ref{lemma:Aw-B-surjective}, we have a surjection $|\bigcup_{w \in W_0} A_w[D']_R| \to |(\im \ov{\Sigma}_{D'})_R|$, where each $A_w[D']_R$ is finite and connected. The image of each $A_w[D']_R$ in $(A/W_0)_R$ is closed, so we see $\pi_0((\im \ov{\Sigma}_{D'})_R) = W_0/{\sim}$, where $\sim$ is the equivalence relation generated by the relation $w \sim w'$ whenever the images of $A_w[D']_R$ and $A_{w'}[D']_R$ in $(A/W_0)_R$ intersect. In this section, we will show that $w \sim w'$ for all $w, w' \in W_0$, and thus $(\im \ov{\Sigma}_{D'})_R$ is connected. In fact, we will see that $A_w[D']_R \cap A_{w'}[D']_R \neq \emptyset$ for all $w, w' \in W_0$ whenever $Z(\hat{G})$ is non-smooth at at most one prime in $D'$. However, this can fail for simple groups in type A by the following example, so this case is worked out separately in Appendix~\ref{appendix:a}.

\begin{example}\label{example:type-A}
    Let $\hat{G} = \SL_6$, let $q = 7$, and let $\alpha$ be a primitive $36$th root of unity. Let $\rm{Fr}$ act trivially on $\hat{G}$, and let $s$ act as $\Ad(t)$, where $t = \diag(\alpha, \alpha, \alpha, \alpha, \alpha, \alpha^{-5})$. If $w = (1\, 2\, 3\, 4\, 5\, 6)$, we claim that $A_1 \cap A_w = \emptyset$. Indeed, if $k$ is a field then any element $(\gamma_1, \dots, \gamma_6)$ of $A'_1(k) \cap A'_w(k)$ is required to satisfy $\gamma_i^6 = \alpha^{-6}$ for all $i$ and also $\gamma_i = \gamma_j$ for all $i, j$. This is not possible for any field $k$ because there is no $6$th root of $\alpha^{-6}$ in $\mu_6(k)$ (as $6$ is not a power of $\chara k$).
\end{example}

Write $s = \Ad(t) \circ s_0$ and $\rm{Fr} = \Ad(t') \circ F_0$, where $t, t' \in \hat{T}(\cO_K[D^{-1}])$ and $s_0$ and $F_0$ preserve a pinning $(\hat{B}, \hat{T}, \{X_\alpha\})$; this is possible by \cite[XXIV, 1.3]{SGA3III}. The relation $\rm{Fr} s \rm{Fr}^{-1} = s^q$ implies $\Ad\left(\prescript{\rm{Fr}}{}{t} \cdot t' \cdot \left(\prescript{s^q}{}{t'}\right)^{-1}\right) = \Ad\left(\prod_{i=0}^{q-1} \prescript{s^i}{}{t}\right)$, so $\prescript{\rm{Fr}}{}{t} \cdot t' \cdot \left(\prescript{s^q}{}{t'}\right)^{-1} = \left(\prod_{i=0}^{q-1} \prescript{s^i}{}{t}\right) \cdot z$ for some central element $z$. In particular, $[\prescript{\rm{Fr}}{}{t}] \cdot [t^{-q}] = [z]$, where we use $[x]$ to denote the image of $x$ in $A$. Since the action is $D$-tame, we have $t^b \in Z(\hat{G})(\cO_K[D^{-1}])$ for some $b \in D'$. Writing $t^b = z_0 z_1$ as a product of central elements, where $z_0$ has order in $D$ and $z_1$ has order in $D'$, we may pass to a translate of $t$ by a $b$th root of $z_0^{-1}$ to assume that $[z]$ is of finite order in $D'$.

\begin{lemma}\label{lemma:reduce-to-central}
    With notation as above, we have $A_w = [t]^{-1} \cdot \left\{a \in A\colon \left(\prescript{\rm{Fr^{-1}}}{}{a}\right)^q\left(\prescript{w}{}{a}\right)^{-1} = \left[\prescript{\rm{Fr}^{-1}}{}{z}^{-1}\right]\right\}$.
\end{lemma}

\begin{proof}
    Let $A'_w = \left\{a \in A\colon \left(\prescript{\rm{Fr^{-1}}}{}{a}\right)^q\left(\prescript{w}{}{a}\right)^{-1} = \left[\prescript{\rm{Fr}^{-1}}{}{z}^{-1}\right]\right\}$. We need only check the equality $A_w = [t]^{-1} \cdot A'_w$ of generically \'etale flat closed subschemes of $A$ on $\ov{\bQ}$-points. By \cite[1.15]{Digne-Michel}, we may choose some $n \in N_{\hat{G}}(\hat{T})(\ov{\bQ})$ lifting $w$ which is fixed by $s_0$. We have then $n\left(\prescript{s}{}{n}\right)^{-1} = ntn^{-1}t^{-1} = \prescript{w}{}{t} \cdot t^{-1}$, so if $a \in A_w(\ov{\bQ})$ then
    \[
    \prescript{\rm{Fr}^{-1}}{}{\left([t]a\right)}^{q} = [t] \cdot \left[\prescript{\rm{Fr}^{-1}}{}{z}^{-1}\right] \cdot \prescript{w}{}{a} \cdot [\prescript{w}{}{t}] \cdot [t^{-1}].
    \]
    This shows $[t]a \in A'_w(\ov{\bQ})$. Moreover, the calculation can be run in reverse to show that if $a' \in A'_w(\ov{\bQ})$ then $[t]^{-1}a' \in A_w(\ov{\bQ})$.
\end{proof}

Lemma~\ref{lemma:reduce-to-central} shows in particular that $\sim$ only depends on $D$, $z$, and th e action of $W_F$ on $A$.

\begin{prop}\label{prop:eigenvalues-connected}
    If $\hat{G}$ is semisimple and simply connected, then $(\im \ov{\Sigma}_{D'})_R$ is connected.
\end{prop}

\begin{proof}
    By \cite[XXIV, 5.3, 5.5]{SGA3III}, there is a unique direct product decomposition $\hat{G} \cong \prod_{i \in I} \hat{G}_i$ for semisimple $\cO_K[D^{-1}]$-group schemes $\hat{G}_i$ with absolutely simple fibers. The action of $W_F$ on $\hat{G}$ permutes the $\hat{G}_i$, so $W_F$ acts on $I$ and we may write $I = \coprod_{j=1}^k I_j$ where each $I_j$ is a single $W_F$-orbit. If $\hat{G}_{I_j} \coloneqq \prod_{i \in I_j} \hat{G}_i$, then we have $\underline{\rm{Z}}^1(W_F^0/P_F, \hat{G})^{r_D = 1} \cong \prod_{j=1}^k \underline{\rm{Z}}^1(W_F^0/P_F, \hat{G}_{I_j})^{r_D = 1}$. Since each $\underline{\rm{Z}}^1(W_F^0/P_F, \hat{G}_{I_j})^{r_D = 1}$ admits a section over $\cO_K[D^{-1}]$ (namely, the trivial cocycle), we may thus reduce from $\hat{G}$ to $\hat{G}_{I_j}$ to assume that $W_F$ permutes the simple factors of $\hat{G}$ transitively. In particular, the simple factors of $\hat{G}$ are pairwise isomorphic.\smallskip
    
    Supppose that the center of $\hat{G}$ is smooth over $\cO_K[(D \cup \{\ell\})^{-1}]$ for some prime $\ell \in D'$, so $[z]$ is of $\ell$-power order. If $\fm$ is a maximal ideal of $R$ containing $\ell$, then Lemma~\ref{lemma:reduce-to-central} shows $t^{-1} \in A_w[D'](R/\fm)$ for all $w \in W_0$, so $W_0/{\sim}$ is a singleton. On the other hand, if $|Z(\hat{G})|$ is divisible by at least two primes in $D'$, then the simple factors of $\hat{G}$ are all isomorphic to $\SL_n$ for some $n$, and thus Lemma~\ref{lemma:type-a-weyl-equivalence} shows that $\pi_0((\im \ov{\Sigma}_{D'})_R) \cong W_0/{\sim}$ is a singleton, as desired.
\end{proof}

\begin{proof}[Proof of Theorem~\ref{theorem:main-general-reduced}]
    By Lemma~\ref{lemma:reduction-sssc}, we may and do pass from $\hat{G}$ to the universal cover of the derived group $\sD(\hat{G})$ to assume that $\hat{G}$ is semisimple and simply connected. By Lemma~\ref{lemma:reduction-fixed-jordan} and Lemma~\ref{lemma:connected-fibers}, the natural map $\pi_0(\underline{\rm{Z}}^1(W_F^0/P_F, \hat{G})_R) \to \pi_0((\im \ov{\Sigma})_R)$ is bijective. By Proposition~\ref{prop:eigenvalues-connected}, the scheme $(\im \ov{\Sigma}_{D'})_R$ is connected, so we are done.
\end{proof}

\begin{remark}
    The proof of Proposition~\ref{prop:eigenvalues-connected} (and thus Theorem~\ref{theorem:main-general}) can be made uniform for $\hat{G}$ such that $Z(\hat{G})_{\cO_K[(D \cup \{\ell\})^{-1}]}$ is smooth for some $\ell \not\in D$, and hence (by splitting into simple factors) if $\hat{G}$ has no simple factors of type A. In the reduction from Theorem~\ref{theorem:main-general} to Theorem~\ref{theorem:main-general-reduced}, one replaces $\hat{G}$ by the identity component $\hat{G}'$ of the group of fixed points of the action of a finite group of order in $D$ on $\hat{G}$, and it may happen that $\hat{G}'$ has simple factors of type A. However, if $Z(\hat{G})_{\cO_K[(D \cup \{\ell\})^{-1}]}$ is smooth for some $\ell \not\in D$, then \cite[3.11]{DHKM} or \cite[3.2, 5.5]{Cotner-Springer} can be used to show that $Z(\hat{G}')_{\cO_K[(D \cup \{\ell\})^{-1}]}$ is also smooth, so a trick using z-extensions and an analogue of \cite[4.27]{DHKM} over $R$ (akin to the argument in Section~\ref{ss:sc}) would allow one to prove Conjecture~\ref{conj:4.3} uniformly away from type A.
\end{remark}

\appendix

\section{Type A}\label{appendix:a}

In this section, we retain the notation of Section~\ref{section:proof}, especially \ref{ss:eigenvalues}. We will assume that $\hat{G} \cong \SL_n^m$ and that $W_F^0/P_F$ permutes the simple factors of $\hat{G}$ transitively. Further, $\hat{T}$ is the diagonal torus and $\hat{B}$ is the upper-triangular Borel. Our aim is to show that $w \sim w'$ for all $w, w' \in W_0$, which will involve constructing a sequence $w = w_0, w_1, \dots, w_N = w'$ such that $A_{w_i} \cap A_{w_{i+1}} \neq \emptyset$ for all $i$. We begin with the first nontrivial example, illustrating part of our algorithm for constructing $w_i$.

\begin{example}\label{example:type-A-good}
    We retain the notation and definitions of Example~\ref{example:type-A}. Although $A_w \cap A_1 = \emptyset$, we will show that $w \sim 1$. Let $w' = (1 \, 2 \, 3)(4 \, 5 \, 6)$, and let $\beta = \alpha^6$. Because a 6th root of unity in $\ov{\bF}_2$ is automatically a 3rd root of unity, a short calculation using Lemma~\ref{lemma:reduce-to-central} shows that the element $t^{-1}\diag(1, \beta, \beta^2, \beta^3, \beta^4, \beta^5)$ lies in $(A_w \cap A_{w'})(\ov{\bF}_2)$. (In fact, $A_w$ and $A_{w'}$ \textit{only} intersect modulo $2$.) If $w'' = (1 \, 2\, 3)$, then in fact $(A_{w'} \cap A_{w''})(\ov{\bQ}) \neq \emptyset$: this follows from a general argument using Lemma~\ref{lemma:reduction-helper} (in the case $a = 1$ and $\sigma_{w} = w$ for all $w$), but one can explicitly compute that $t^{-1}\diag(\alpha^{11}, \alpha^{-1}, \alpha^{23}, \alpha, \alpha, \alpha)$ lies in this intersection. Moreover, $1 \in (A_{w''} \cap A_1)(\ov{\bQ})$. This gives a chain which proves $w \sim 1$ in this case.
\end{example}

Example~\ref{example:type-A-good} gives the model for our argument below: break apart $n$-cycles by specializing modulo a well-chosen prime, then cut off individual cycles by finding intersections over $\ov{\bQ}$. In order to deal with actions of $W_F^0/P_F$ which are more complicated than the one in Example~\ref{example:type-A-good}, the general argument requires more notation, which we now begin to introduce.\smallskip

Let $I_1, \dots, I_a$ be the $s$-orbits on $\{1, \dots, m\}$. Since $W_F^0/P_F$ permutes the simple factors of $\hat{G}$ transitively, it follows that $\rm{Fr}$ permutes the $I_i$ transitively. By reordering the $I_i$, we may and do assume that $\rm{Fr}$ sends each $I_{i+1}$ to $I_i$, where we write $I_{a + 1} = I_1$. We first note that we may reduce to the case $a = 1$.

\begin{lemma}\label{lemma:shapiro}
    Let $F_a$ denote the unramified extension of $F$ of degree $a$. The natural restriction map $r\colon \underline{\rm{Z}}^1(W_F^0, \hat{G}) \to \underline{\rm{Z}}^1(W_{F_a}^0, \hat{G}_{I_1})$ is smooth with geometrically connected fibers.
\end{lemma}

\begin{proof}
    The proof of \cite[5.13]{DHKM} can be read directly with only very minor adaptations to show that $r$ is isomorphic to the projection map of $\underline{\rm{Z}}^1(W_{F_a}^0, \hat{G}_{I_1}) \times \hat{G}_{I_1}^{a-1}$ onto its first factor, where $\hat{G}_{I_1}$ is the product of the factors of $\hat{G}$ with index in $I_1$.
\end{proof}

Note that $s^m$ induces an automorphism of each factor of $\hat{G}$; we denote its image in the group of outer automorphisms of $\SL_n$ by $\eps(s)$, and note that it is independent of the choice of factor. The relation $\rm{Fr} s \rm{Fr}^{-1} = s^q$ implies that $\eps(s) = \eps(s)^q$, so $\eps(s)$ is trivial if $q$ is even. If $S$ denotes the diagonal torus of $\SL_n$, then $\eps(s)$ induces an automorphism of $S$ by identifying $S$ as the group of coinvariants of $S^m$ under the permutation action induced by $s$. We have $A = \hat{T}/(1 - s)\hat{T} = S/(1 - \eps(s))S$. Similarly, $\rm{Fr}$ induces an automorphism of $S$, and we denote its image in the group of outer automorphisms by $\eps(\rm{Fr})$. When the context makes it unambiguous, we will also use $\eps(s)$ and $\eps(\rm{Fr})$ to denote their images under the canonical isomorphism $\Out(\SL_n) \cong \{\pm 1\}$.\smallskip

The action of $\rm{Fr}$ on $S$ is given by $(\beta_1, \dots, \beta_n) \mapsto \left(\beta_{\rm{Fr}(1)}^{\eps(\rm{Fr})}, \dots, \beta_{\rm{Fr}(n)}^{\eps(\rm{Fr})}\right)$, where we use $\rm{Fr}$ here and below to (abusively) denote the involution of $\{1, \dots, n\}$ induced by $\eps(\rm{Fr})$. Similarly, we will use $s$ to denote the permutation induced by $\eps(s)$. If $t = (\gamma_j)_{j=1}^n \in S(A)$ for a ring $A$, and $\sigma \in S_n$ is a permutation, then by definition $\prescript{\sigma}{}{t} \coloneqq (\gamma_{\sigma^{-1}(j)})_{j=1}^n$.\smallskip

Keeping the notation preceding Lemma~\ref{lemma:reduce-to-central}, we have $s = \Ad(t) \circ s_0$ and $\rm{Fr} = \Ad(t') \circ F_0$, where $t, t' \in \hat{T}(\cO_K[D^{-1}])$ and $s_0$ and $F_0$ preserve a pinning $(\hat{B}, \hat{T}, \{X_\alpha\})$, and we have
\[
\prescript{\rm{Fr}}{}{t} \cdot t' \cdot \left(\prescript{s^q}{}{t'}\right)^{-1} = \left(\prod_{i=0}^{q-1} \prescript{s^i}{}{t}\right) \cdot z
\]
for a central element $z$. Note that the map $Z(\hat{G})^s \to Z(\hat{G})^s$ given by $t \mapsto \prescript{\rm{Fr}}{}{t} \cdot t^{-q}$ has kernel of order dividing $q^a - \eps(\rm{Fr})$, as one can see by a straightforward calculation. Thus, after possibly extending $K$, we may and do alter $t$ by an element of $Z(\hat{G})^s(\cO_K[D^{-1}])$ to assume that every prime dividing the order of $z$ divides $q^a - \eps(\rm{Fr})$. Similarly, since the action of $W_F$ is $D$-tame, we may and do assume that $z$ has order in $D'$.\smallskip

If $w \in W_0$, we define $S_w$ to be the closed subscheme of $S$ consisting of those sections $t$ such that $\left(\prescript{\rm{Fr}^{-1}}{}{t}\right)^q \cdot \left(\prescript{w}{}{t}\right)^{-1} = \left[\prescript{\rm{Fr}^{-1}}{}{z}^{-1}\right]$. By Lemma~\ref{lemma:reduce-to-central}, if $S_w[b] \cap S_{w'}[b] \neq \emptyset$ for some $b \in D'$, then $w \sim w'$. In fact, it is enough to show $S_w \cap S_{w'}$ admits a torsion point: suppose $k$ is a field over $R$ and $t \in S_w(k) \cap S_{w'}(k)$ is of order $n = ab$ for $a \in D$, $b \in D'$. Let $c \in D$ be such that $ac \equiv 1 \pmod{n_0}$, and note
\[
\left(\prescript{\rm{Fr}^{-1}}{}{t}^{ac}\right)^q \cdot \left(\prescript{w}{}{t}^{ac}\right)^{-1} = \left[\prescript{\rm{Fr}^{-1}}{}{z}^{-1}\right]^{ac} = \left[\prescript{\rm{Fr}^{-1}}{}{z}^{-1}\right],
\]
so $t^{ac} \in S_w[b](k) \cap S_{w'}[b](k)$. For convenience, if $w \in W_0$ we will write $\sigma_w \coloneqq \rm{Fr}w \in W_0$. \smallskip

Having set up the necessary notation, we are now ready to begin proving that $w \sim w'$ for all $w, w' \in W_0$. By symmetry, we need only show that $w \sim \rm{Fr}$ for all $w \in W_0$.

\begin{lemma}\label{lemma:reduction-helper}
    Let $w, w' \in W_0$ be such that $\sigma_w$ commutes with $\sigma_{w'}$. If $S^{\sigma_w^{-1}\sigma_{w'}}$ is a torus, then $S_w[b] \cap S_{w'}[b] \neq \emptyset$ for some $b \in D'$.
\end{lemma}

\begin{proof}
    Define a map $f\colon S \to S$ by $f(t) = t^q \cdot \left(\prescript{\sigma_w}{}{t}\right)^{-1}$. Let $k$ be an algebraically closed field over $R$ and note $(\ker f)(k)$ has no $p$-torsion, so $f$ is an isogeny. If $A$ is an $R$-algebra and $t \in S(A)$ is fixed by $\sigma_w^{-1}\sigma_{w'}$, then we have
    \[
    \prescript{\sigma_w^{-1}\sigma_{w'}}{}{f(t)} = \prescript{\sigma_w^{-1}\sigma_{w'}}{}{\left(t^q \cdot \left(\prescript{\sigma_w}{}{t}\right)^{-1}\right)} = f(t)
    \]
    since $\sigma_w$ commutes with $\sigma_{w'}^{-1} \sigma_w = w^{-1}w'$. Thus $f$ induces an isogeny of tori $S^{\sigma_w^{-1}\sigma_{w'}} \to S^{\sigma_w^{-1}\sigma_{w'}}$. In particular, since $Z(\SL_n) \subset S^{\sigma_w^{-1}\sigma_{w'}}$, it follows that there exists a torsion element $t \in S^{\sigma_w^{-1}\sigma_{w'}}(k)$ such that $t \in S_w(k)$. But $\prescript{w}{}{t} = \prescript{w'}{}{t}$, so also $t \in S_{w'}(k)$, and thus $w \sim w'$.
\end{proof}

Lemma~\ref{lemma:reduction-helper} is not enough to prove $w \sim w'$ for all $w, w' \in W_0$, but it is enough to show this for ``most" $w$ and $w'$, as we now explain. We will say $\sigma \in S_n$ is an \textit{$s$-cycle} if either $\sigma$ is a cycle or there is a cycle $c \in S_n$ such that $c$ and $\prescript{\eps(s)}{}{c}$ are disjoint and $\sigma = c \cdot \prescript{\eps(s)}{}{c}$. If $\sigma \in S_n$ is fixed by $\eps(s)$, then there is a unique decomposition $\sigma = c_1 \cdots c_r$ of $\sigma$ into disjoint $s$-cycles $c_1, \dots, c_r$. If $\sigma$ is an $s$-cycle, then we define its \textit{length} to be the number of $i$ such that $\sigma(i) \neq i$.

\begin{lemma}\label{lemma:reduce-to-trivial}
    If $\sigma_w$ is not a single $s$-cycle of length $n$, then $S_w[b] \cap S_{\rm{Fr}}[b] \neq \emptyset$ for some $b \in D'$.
\end{lemma}

\begin{proof}
    First, let $\sigma_w = c_1 \cdots c_r$ be the decomposition of $w$ into disjoint $s$-cycles of length $> 1$, so either $r \geq 2$ or $r=1$ and $c_1$ is of length $< n$. Note $w \sim \rm{Fr}c_1$ by Lemma~\ref{lemma:reduction-helper}, since $S^{\sigma}$ is a torus for any permutation $\sigma$ in $S_n$ which has a fixed point. Thus we may assume that $r = 1$. In this case, another application of Lemma~\ref{lemma:reduction-helper} shows that $w \sim \rm{Fr}$, as desired.
\end{proof}

Thus we must show that if $w \in W_0$ is a single $s$-cycle of length $n$, then there is some $w' \in W_0$ which is not a single $s$-cycle of length $n$ such that $w \sim w'$. We begin with an elementary lemma for which we do not know a reference.

\begin{lemma}\label{lemma:elementary-congruence}
    Let $n_0, e, q \geq 1$ be integers such that every prime number dividing $n_0$ divides $q - 1$.
    \begin{enumerate}
        \item $\sum_{i = 0}^{n_0 e - 1} q^i \equiv 0 \pmod{n_0}$,
        \item $\sum_{i = 0}^{n_0 e - 1} (n_0 e - 1 - i) q^i \equiv 0 \pmod{n'_0}$, where $n_0' = n_0$ for $n_0$ odd and $n_0' = n_0/2$ for $n_0$ even.
    \end{enumerate}
\end{lemma}

\begin{proof}
    We may and do assume $e = 1$ and $n_0 = \ell^a$ for some prime number $\ell$ and some $a \geq 1$. Write $q = \ell^c d + 1$, where $\ell \nmid d$. For (1), note that
    \[
    \sum_{i=0}^{\ell^a - 1} q^i = \frac{(\ell^c d + 1)^{\ell^a} - 1}{\ell^c d}.
    \]
    It is elementary to check that if $x \equiv y \pmod{\ell^c}$, then $x^{\ell^a} \equiv y^{\ell^a} \pmod{\ell^{a+c}}$, so the displayed equation implies (1). Note that the sum in (2) is equal to the value at $X = \ell^c d$ of 
    \[
    \frac{\rm{d}}{\rm{d}x} \frac{1 - (1 + X)^{\ell^a}}{X} = \frac{-\ell^a X (1 + X)^{\ell^a - 1} + (1 + X)^{\ell^a} - 1}{X^2} = \sum_{i=2}^{\ell^a} \left({\ell^a \choose i} - \ell^a {\ell^a - 1 \choose i - 1}\right) X^{i - 2}.
    \]
    Modulo $\ell^a$, this sum is equal to $\frac{q^{\ell^a} - (1 + \ell^{a+c} d)}{q^2}$. Standard estimates using Legendre's formula show that $q^{\ell^a} \equiv 1 + \ell^{a+c} d \pmod{\ell^{a+2c}}$ if $\ell \neq 2$, and $q^{\ell^a} \equiv 1 + \ell^{a+c} d \pmod{\ell^{a + 2c - 1}}$ if $\ell = 2$. This immediately yields (2).
\end{proof}

An element of $S$ is identified with a tuple $(\gamma_j)$ with $1 \leq j \leq n$ such that $\prod_{j=1}^n \gamma_j = 1$. Note in particular that $\left[\prescript{\rm{Fr}^{-1}}{}{z}^{-1}\right]$ is identified with an $n$th root of unity $\alpha$. Note that a tuple $(\gamma_j)$ lies in $S_w$ if and only if
\[
\gamma_{\sigma_w^{-1}(j)} = \gamma_j^{\eps(\rm{Fr})q} \alpha^{-1}
\]
for all $j$.

\begin{lemma}\label{lemma:reduce-from-length-n}
    Suppose $w$ is an $s$-cycle of length $n$. There is some $w' \in W_0$ which is a product of disjoint $s$-cycles of length $< n$ such that $S_w[b] \cap S_{w'}[b] \neq \emptyset$ for some $b \in D'$.
\end{lemma}

\begin{proof}
    If $n$ is divisible by at most one prime in $D'$, then this follows from the argument of Proposition~\ref{prop:eigenvalues-connected}, so we will assume that $n$ is divisible by at least two primes in $D'$. Let $n'$ equal $n$ if $w$ is a single cycle of length $n$, and $n/2$ if $w$ is a product of two disjoint $s$-conjugate cycles whose lengths add to $n$ (so in particular $n$ is even and $2 \in D'$). Let $\ell$ be a prime number in $D'$ dividing $n'$, and write $n' = \ell^c d$ for some $d$ such that $\ell \nmid d$. Define $w' \in W_0$ as follows: $\sigma_{w'}\sigma_w^i(1)$ is equal to $\sigma_w^{i+1}(1)$ if $d \nmid i+1$, and $\sigma_w^{i+1 - d}(1)$ if $d \mid i + 1$. If $n' = n$, then this determines $w'$; otherwise, we also require that $\sigma_{w'}\sigma_w^i(n)$ is equal to $\sigma_w^{i+1}(n)$ if $d \nmid i+1$, and $\sigma_w^{i+1 - d}(n)$ if $d \mid i+1$.\smallskip
    
    If $n' = n$, then we find
    \[
    \sigma_{w'} = (1 \,\sigma_w(1) \cdots \sigma_w^{d-1}(1))(\sigma_w^d(1) \cdots ) \cdots (\sigma_w^{(\ell^c - 1)d}(1) \cdots \sigma_w^{n - 1}(1)).
    \]
    If $n' = n/2$, then we have
    \begin{align*}
    \sigma_{w'} = (1 \,\sigma_w(1) \cdots \sigma_w^{d-1}(1))\cdots &(\sigma_w^{(\ell^c - 1)d}(1) \cdots \sigma_w^{n' - 1}(1)) \cdot \\
        \cdot &(n \cdots \sigma_w^{d - 1}(n)) \cdots (\sigma_w^{(\ell^c - 1)d}(n) \cdots \sigma_w^{n'-1}(n)).
    \end{align*}
    We will show that, if one of the conditions
    \begin{enumerate}
        \item $n' = n$ is odd,
        \item $n' = n$ and $2 \in D$,
        \item $n' = n$ and $\ell = 2 \in D'$,
        \item $n' = n/2$ is even and $\ell = 2 \in D'$,
        \item $n' = n/2$ is odd and $\ell \not\in D \cup \{2\}$,
    \end{enumerate}
    holds, then $w' \in W_0$ and $(S_w \cap S_{w'})(\ov{\bF}_\ell) \neq \emptyset$, and thus $w \sim w'$. This is enough to prove the lemma: if $n' = n/2$, then $2 \in D'$ and this follows from (4) and (5) because $n$ is divisible by at least two primes in $D'$ by assumption. If $n' = n$ is odd, then this follows from (1). If $n' = n$ is even, then this follows from (2) and (3) except possibly if $\ell = 2 \in D'$ and $n/2$ is odd, in which case we can pass from $w$ to $w'$ to reduce to the already-established case $n' = n/2$. \smallskip

    We begin by showing $w' \in W_0$. If $\eps(s) = 1$, then this is trivial, so suppose $\eps(s) = -1$. In this case, it follows that $n$ is even, since any permutation of $\{1, \dots, 2m+1\}$ fixed by the flip fixes $m+1$. Moreover, $2 \in D'$. We have
    \[
    \eps(s)\sigma_{w'}\sigma_w^i(1) =
    \begin{cases}
        &\sigma_w^{i+1}(n) \text{ if } d \nmid i+1, \\
        &\sigma_w^{i+1 - d}(n) \text{ if } d \mid i+1.
    \end{cases}
    \]
    On the other hand, if $n' = n$ then $\eps(s)(1) = n = \sigma_w^{n/2}(1)$ and thus
    \[
    \sigma_{w'}\sigma_w^i(\eps(s)(1)) =
    \begin{cases}
        &\sigma_w^{i + 1}(n) \text{ if } d \nmid i+1+\frac{n}{2}, \\
        &\sigma_w^{i + 1 - d}(n) \text{ if } d \mid i+1+\frac{n}{2}.
    \end{cases}
    \]
    If $\ell = 2$, then $d \mid \frac{n}{2}$, so we find $\eps(s)\sigma_{w'} = \sigma_{w'}\eps(s)$, i.e., $\sigma_{w'} \in W_0$, and thus $w' \in W_0$ since $\rm{Fr} \in W_0$. Thus for the rest of this proof, we will assume $\ell = 2$ whenever $\eps(s) = -1$ and $n' = n$. If instead $n' = n/2$, then the definitions immediately imply $\eps(s) \sigma_{w'} = \sigma_{w'} \eps(s)$ without restriction on $\ell$.\smallskip

    To show $w \sim w'$, suppose first that $n' = n$. If either $n$ is odd or $2 \in D$, then we choose $\ell \in D'$ arbitrarily; if $n$ is even and $2 \in D'$, then we take $\ell = 2$. By the congruences of Lemma~\ref{lemma:elementary-congruence}, there is an element $t = (\gamma_j)$ of $S_w(\ov{\bF}_\ell)$ determined by $\gamma_{\sigma_w^{-j}(1)} = \prod_{i=0}^{j-1} \alpha^{-(\eps(\rm{Fr})q)^i}$ for all $j$. The key observation is now that $\gamma_{\sigma_w^j(1)} = \gamma_{\sigma_w^{j + d}(1)}$ in $\ov{\bF}_\ell$ for all $j$: indeed, unraveling, we must show that $\prod_{i=0}^{d-1} \alpha^{\eps(\rm{Fr})q^i} = 1$ in $\ov{\bF}_\ell$. This follows from Lemma~\ref{lemma:elementary-congruence}(1), using the facts that $\alpha$ is a $\gcd(d, n_0)$th root of unity in $\ov{\bF}_\ell$ and that $\eps(\rm{Fr})q \equiv 1 \pmod{\ell'}$ for every prime $\ell'$ dividing $\gcd(d, n_0)$. This shows $t \in S_{w'}(\ov{\bF}_\ell)$. Since $t$ is torsion, we find $w \sim w'$. \smallskip

    Next, suppose $n' = n/2$. In this case, $n$ is even and $2 \in D'$. If either $\ell = 2$, or $n'$ is odd and $\ell \not\in D \cup \{2\}$, then we let $t = (\gamma_j)$ be the element of $S_w(\ov{\bF}_\ell)$ determined by $\gamma_{\sigma_w^{-j}(1)} = \gamma_{\sigma_w^{-j}(n)} = \prod_{i=0}^{j-1} \alpha^{-(\eps(\rm{Fr})q)^i}$ for all $j \geq 0$. Again, the congruences of Lemma~\ref{lemma:elementary-congruence} imply that this is well-defined, and the same argument as before implies $t \in S_{w'}(\ov{\bF}_\ell)$, allowing us to conclude.
\end{proof}

\begin{lemma}\label{lemma:type-a-weyl-equivalence}
    Suppose $\hat{G} \cong \SL_n^m$, where $W_F$ permutes the simple factors transitively. Then $w \sim w'$ for all $w, w' \in W_0$.
\end{lemma}

\begin{proof}
    It follows from Lemma~\ref{lemma:reduce-from-length-n} and Lemma~\ref{lemma:reduce-to-trivial} that $w \sim \rm{Fr}$ for all $w \in W_0$, as desired.
\end{proof}

\bibliographystyle{halpha-abbrv}
\bibliography{bibliography}

\end{document}